\documentclass[11pt]{amsart}
\usepackage[hmargin=32mm, vmargin=27mm, includefoot, twoside]{geometry}
\usepackage[colorlinks,bookmarksopen=true]{hyperref}
\usepackage[latin1]{inputenc}  % accents 8 bits dans la source
\usepackage[T1]{fontenc}       % accents dans le DVI
\usepackage[polutonikogreek, english]{babel}
\usepackage{amssymb}
\usepackage{latexsym}
\usepackage{mathrsfs}
\usepackage{xspace}
\usepackage[all, cmtip]{xy}
\usepackage{xr}
\newcommand{\dref}[1]{\ref{discrete-#1} in~\cite{Caprace-Monod_discrete}}%        reference à "discrete"
\newtheorem{prop}{Proposition}[section]
\newtheorem{thm}[prop]{Theorem}
\newtheorem*{thm*}{Theorem}
\newtheorem{addendum}[prop]{Addendum}
\newtheorem*{addendum*}{Addendum}
\newtheorem{cor}[prop]{Corollary}
\newtheorem{lem}[prop]{Lemma}

\newtheorem*{convention*}{Convention}
\theoremstyle{definition}
\newtheorem*{defn*}{Definition}

\newtheorem{remark}[prop]{Remark}
\newtheorem{remarks}[prop]{Remarks}

\newtheorem*{scholium*}{Scholium}
\theoremstyle{remark}
\newtheorem{example}[prop]{Example}
\newtheorem*{example*}{Example}
\numberwithin{equation}{section}

\newcommand{\sepa}{\nobreak\vskip1.5mm\nobreak%
\begin{center}%
\vrule height 1pt width.2\hsize%
\,.\,%
\vrule height 1pt width.2\hsize%
\end{center}%
\vskip5mm}

\newcommand{\vareps}{\varepsilon}

\newcommand{\ro}{\varrho}
\newcommand{\teta}{\vartheta}

\newcommand{\GG}{\mathbf{G}}

\newcommand{\LL}{\mathbf{L}}
\newcommand{\NN}{\mathbf{N}}
\newcommand{\PP}{\mathbf{P}}
\newcommand{\QQ}{\mathbf{Q}}
\newcommand{\RR}{\mathbf{R}}
\newcommand{\TT}{\mathbf{T}}
\newcommand{\ZZ}{\mathbf{Z}}

\newcommand{\SL}{\mathrm{SL}}

\newcommand{\la}{\langle}
\newcommand{\ra}{\rangle}
\newcommand{\inv}{^{-1}}
\newcommand{\opp}{^\mathrm{op}}

\newcommand{\norma}{\mathscr{N}}
\newcommand{\centra}{\mathscr{Z}}

\newcommand{\QZ}{\mathscr{Q\!Z}}
\newcommand{\se}{\subseteq}

\newcommand{\wt}{\widetilde}
\newcommand{\lra}{\longrightarrow}

\def\bs#1.{
              \def\temp{#1}
              \ifx\temp\empty
                   \mathcal{B}
              \else
                   \mathcal{B}(#1)
              \fi
}
\newcommand{\cat}{{\upshape CAT(0)}\xspace}
\newcommand{\tangle}[2]% angle de Tits
{\angle_\mathrm{T}(#1,#2)}
\newcommand{\aangle}[3]% angle d'Alexandrov
{\angle_{#1}(#2,#3)}
\newcommand{\cangle}[3]% angle de comparaison
{\overline{\angle}_{#1}(#2,#3)}
\DeclareMathOperator{\Stab}{Stab}

\DeclareMathOperator{\Isom}{Is}

  \DeclareMathOperator{\soc}{soc}
\newcommand{\bd}{\partial} % mathoperator ajoute beaucoup d'espace....

\def\Sym{\mathop{\mathrm{Sym}}\nolimits}
\def\min{\mathop{\mathrm{min}}\nolimits}

\begin{document}
\title[Isometry groups of non-positively curved spaces: structure theory]{Isometry groups of non-positively curved spaces:\\ structure theory}
\author[Pierre-Emmanuel Caprace]{Pierre-Emmanuel Caprace*}
\address{UCL, 1348 Louvain-la-Neuve, Belgium}
\email{pe.caprace@uclouvain.be}
\thanks{*F.N.R.S. Research Associate}
\author[Nicolas Monod]{Nicolas Monod$^\ddagger$}
\address{EPFL, 1015 Lausanne, Switzerland}
\email{nicolas.monod@epfl.ch}
\thanks{$^\ddagger$Supported in part by the Swiss National Science Foundation}
%\date{January 2009}%  ``\today'' est dangereux si on met le papier sur Arxiv, voir http://arxiv.org/help/faq/today
%\subjclass{20F65; 20E42, 20F50, 43A07} % AMS classification numbers
\keywords{Non-positive curvature, \cat space, locally compact group, lattice}
\begin{abstract}
We develop the structure theory of full isometry groups of locally compact non-positively curved metric spaces.
Amongst the discussed themes are de Rham decompositions, normal subgroup structure and characterising properties
of symmetric spaces and Bruhat--Tits buildings. Applications to discrete groups and further developments on
non-positively curved lattices are exposed in a companion paper~\cite{Caprace-Monod_discrete}.
\end{abstract}
\maketitle
\let\languagename\relax  % TO FIX A BUG IN RUNNING HEADERS AND BABEL

%%%%%%%%%%%%%%%%%%%%%%%%%%%%%%%%%%%%%%%%%%%%%%%%%%%%%%%%%%%%%%%%%%%%%%%%%%%%%%
\section{Introduction}
%%%%%%%%%%%%%%%%%%%%%%%%%%%%%%%%%%%%%%%%%%%%%%%%%%%%%%%%%%%%%%%%%%%%%%%%%%%%%%

Non-positively curved metric spaces were introduced by A.~D.~Alexandrov~\cite{Alexandrov} and popularised by
M.~Gromov, who called them \cat spaces. Their theory offers a wide gateway to a form of
generalised differential geometry, whose objects encompass Riemannian manifolds of non-positive sectional
curvature as well as large families of singular spaces including Euclidean buildings and many other polyhedral
complexes. It has found a wide range of applications to various fields, including semi-simple algebraic and
arithmetic groups, and geometric group theory.

A recurrent theme in this area is the interplay between the geometry of a locally compact
\cat space $X$ and the algebraic properties of a discrete group $\Gamma$ acting properly on $X$ by isometries.
This interaction is expected to be especially rich and tight when the $\Gamma$-action is cocompact; the pair
$(X, \Gamma)$ is then called a \textbf{\cat group}\index{CAT(0) group@\cat group}. The purpose of the present paper
and its companion~\cite{Caprace-Monod_discrete} is to highlight the rôle of a third entity through which
the interaction between $X$ and $\Gamma$ transits: namely the full isometry group $\Isom(X)$ of
$X$. The topology of uniform convergence on compacta makes $\Isom(X)$ a locally compact second countable group
which is thus canonically endowed with Haar measures. It therefore makes sense to consider lattices in
$\Isom(X)$, \emph{i.e.} discrete subgroups of finite invariant covolume; we call such pairs $(X, \Gamma)$
\textbf{\cat lattices}\index{CAT(0) lattice@\cat lattice} (thus \cat groups are precisely \emph{uniform} \cat
lattices). This immediately suggests the following two-step programme:

\begin{itemize}
\item[(I)] To develop the basic structure theory of the locally compact group $\Isom(X)$ and deduce consequences
on the overall geometry of the underlying proper \cat space $X$. This is the main purpose of the present paper.

\item[(II)] To study \cat lattices and thus in particular \cat groups by building upon the structure results
of the present paper, using new geometric density and superrigidity techniques. This is carried out in
the subsequent paper~\cite{Caprace-Monod_discrete}.
\end{itemize}

We now proceed to describe the main results of this first part in more detail. First, in~\S\,\ref{sec:intro:complete},
we present results in the special case of \textbf{geodesically complete}\index{geodesically complete} \cat spaces,
\emph{i.e.} spaces in which every geodesic segment can be extended to a bi-infinite geodesic line~--- which need
not be unique. Important examples of geodesically complete spaces are provided by Bruhat--Tits buildings and of
course Hadamard manifolds, \emph{e.g.} symmetric spaces.

\smallskip

The second and longer part of the Introduction, \S\,\ref{sec:intro:general}, will present results valid for arbitrary locally compact \cat spaces.
In either case, the entire contents of the Introduction rely on more general, more detailed but probably also more cumbrous statements
proved in the core of the text.

%================================================================================
\subsection{Spaces with extensible geodesics}\label{sec:intro:complete}
The conclusions of several results become especially clear and perhaps more striking in the special case of
geodesically complete \cat spaces. Beyond Euclidean buildings and Hadamard manifolds, we recall that a
complete \cat space that is also a homology manifold
has automatically extensible geodesics~\cite[II.5.12]{Bridson-Haefliger}. Note also that it is always possible
to artificially make a \cat space geodesically complete by gluing rays, though it is not always possible to
preserve properness (consider a compact but total set in an infinite-dimensional Hilbert space).

\bigskip

%================================================================================
\noindent\textbf{Decomposing \cat spaces into products of symmetric spaces and locally finite cell complexes.}
Prototypical examples of locally compact \cat spaces are mainly provided by the following two  sources.
\begin{itemize}
\item[---] Riemannian manifolds of non-positive sectional curvature, whose most prominent
representatives are the Riemannian symmetric spaces of non-compact type. These spaces are \textbf{regular} in
the sense that any two geodesic segments intersect in at most one point. The full isometry group of such a space
is a Lie group.

\item[---] Polyhedral complexes of piecewise constant non-positive curvature, such as trees or Euclidean
buildings. These spaces are \textbf{singular} in the sense that geodesics do branch. The subgroup of the
isometry group which preserves the cell structure is totally disconnected.
\end{itemize}

The following result seems to indicate that a \cat space often  splits as a product of spaces belonging to these
two families.

\begin{thm}\label{thm:GeodesicallyComplete}
Let $X$ be a proper geodesically complete \cat space whose isometry group acts cocompactly without fixed point
at infinity. Then $X$ admits an $\Isom(X)$-equivariant splitting
$$X = M \times \RR^n \times Y,$$
where $M$ is a symmetric space of non-compact type and the isometry group $\Isom(Y)$ is totally disconnected
and acts by semi-simple isometries on $Y$ (each factor may be trivial).

Furthermore, the space $Y$ admits an $\Isom(Y)$-equivariant locally
finite decomposition into convex cells, where the cell supporting a point $y \in Y$ is defined as the fixed
point set of the isotropy group $\Isom(Y)_y$.
\end{thm}

If $X$ is regular, then $\Isom(Y)$ is discrete. In other words, the space $Y$ has branching geodesics as soon as
$\Isom(Y)$ is non-discrete. We refer to Theorem~\ref{thm:Decomposition} and Addendum~\ref{addendum} below for a
version of the above without the assumption of extensibility of geodesics.

%Specialising to \cat cell complexes, one obtains the following clarification on the distinction between the
%\emph{full} isometry group and the subgroup preserving the cell structure.???
%
%\begin{cor}\label{cor:CellPreserving}
%Let $C$ be a finite non-positively curved cell complex without boundary and let $X$ be its universal cover. If
%$X$ is irreducible???, then either $X$ is a symmetric space or the full isometry group $\Isom(X)$ preserves a
%subdivision of the given cell decomposition.
%\end{cor}

We emphasize that the `cells' provided by Theorem~\ref{thm:GeodesicallyComplete} need not be compact; in fact if
$\Isom(Y)$ acts freely on $Y$ then the decomposition in question becomes trivial and consists of a single cell,
namely the whole of $Y$. Conversely the cell decomposition is non-trivial provided $\Isom(Y)$ does not act
freely. The most obvious way for the $\Isom(Y)$-action  not to be free is if $\Isom(Y)$ is not discrete. A
strong version of the latter condition is that \emph{no open subgroup of fixes a point at infinity}; this holds
notably for symmetric spaces and Bruhat--Tits buildings. A quite immediate consequence of this condition is that
the above cells are then necessarily compact. We shall show that much additional structure can be derived from
it (see Section~\ref{sec:NoOpenStabiliser} below).

\sepa

%================================================================================
\noindent\textbf{Smoothness.} The cell decomposition of the third factor in Theorem~\ref{thm:GeodesicallyComplete} is
derived from the following \textbf{smoothness}\index{smooth} result for isometric actions of totally
disconnected groups.

\begin{thm}\label{thm:smooth:intro}
Let $X$ be a geodesically complete proper \cat space $X$ and $G<\Isom(G)$ a totally disconnected (closed)
subgroup acting minimally.

The the pointwise stabiliser in $G$ of every bounded set is open.
\end{thm}

This property, which is familiar from classical examples, \emph{does in general fail} without geodesic completeness
(see Remark~\dref{rem:non-smooth}).
It is an important ingredient for the considerations of Section~\ref{sec:NoOpenStabiliser} alluded to above, as
well as for angle rigidity results regarding both the Alexandrov angle (Proposition~\ref{prop:angle_rigidity})
and the Tits angle (Proposition~\ref{prop:discrete:orbits}).

\sepa

%================================================================================
\noindent\textbf{A characterisation of symmetric spaces and Euclidean buildings.} In symmetric spaces and
Bruhat--Tits buildings, the stabilisers of points at infinity are exactly the parabolic subgroups; as such, they
are cocompact. This cocompactness holds further for all Bass--Serre trees, namely bi-regular trees. Combining
our results with work of B.~Leeb~\cite{Leeb} and A.~Lytchak~\cite{Lytchak_RigidityJoins}, we establish a
corresponding characterisation.

\begin{thm}\label{thm:cocompact:stabilisers}
Let $X$ be a geodesically complete proper \cat space. Suppose that the stabiliser of every point at infinity
acts cocompactly on $X$.

Then $X$ is isometric to a product of symmetric spaces, Euclidean buildings and Bass--Serre trees.
\end{thm}

The Euclidean buildings appearing in the preceding statement admit an automorphism group that is strongly
transitive, \emph{i.e.} acts transitively on pairs $(c, A)$ where $c$ is a chamber and $A$ an apartment
containing $c$. This property characterises the Bruhat--Tits buildings, except perhaps for some two-dimensional
cases where this is a known open question.

\smallskip

The above characterisation is of a different nature and independent of the characterisations using lattices that
will be presented in~\cite{Caprace-Monod_discrete}.

\sepa

%================================================================================
\noindent\textbf{Actions of simple algebraic groups.} Both for the general theory and for the geometric
superrigidity/arithmeticity statements of the companion paper~\cite{Caprace-Monod_discrete}, it is important to understand how
algebraic groups act on \cat spaces.

\begin{thm}\label{thm:algebraic:intro}
Let $k$ be a local field and $\GG$ be an absolutely almost simple simply connected $k$-group. Let $X$ be a
\cat space (not reduced to a point) on which $G = \GG(k)$ acts continuously and cocompactly by isometries.

Then $X$ is isometric to $X_{\mathrm{model}}$, the Riemannian symmetric space or Bruhat--Tits building associated with $G$.
\end{thm}

A stronger and much more detailed statement is provided below as Theorem~\ref{thm:algebraic}. In particular, a modification
of the statement holds without extensibility of geodesics and the cocompactness assumption can be relaxed. However, we also
show there by means of two examples that some assumptions remain necessary.

(As a point of terminology, we do not choose a particular scaling factor on $X_{\mathrm{model}}$,
so that the isometry of our statement could become a homothety for another model.)

\sepa

%================================================================================
\subsection{General case}\label{sec:intro:general}
%================================================================================
 When dealing with \cat space in the highest possible level of generality, one has
to face several technical difficulties caused by the flexibility of the \cat condition. For example, given a
\cat space $X$, there are many ways to deform it in order to construct another space $Y$, non-isometric to $X$,
but with the property that $X$ and $Y$ have isomorphic isometry groups or/and identical boundaries. Amongst the
simplest constructions, one can form (possibly warped) products with compact \cat spaces or grow hair
equivariantly along a discrete orbit. Much wilder (non-quasi-isometric) examples can be constructed for instance
by considering warped products with the very vast family of \cat spaces having no isometries and a unique point
at infinity.

In order to address these issues, we introduce the following terminology.

\smallskip

\noindent\textbf{Minimality.}
\begin{flushright}
\begin{greektext}
\itshape
\ldots <'osa tis >`an e>'ipoi sfa'iras >egk'wmia, a>ut`a ta\~uta ka`i fal'akras >egk'wmia diex'erqetai.
\upshape

\smallskip

Sun'esios Kurena'iou, \emph{Fal'akras >egk'wmion}%
\end{greektext}%
.\footnote{Synesius of Cyrene,
\begin{greektext}
\emph{Fal'akras >egk'wmion}
\end{greektext}
(known as \emph{Calvitii encomium}), end of Chapter~8 (at 72A in the page numbering from Denis Pétau's 1633 edition).
The \emph{Encomium} was written around 402; we used the 1834 edition by J.~G.~Krabinger (Ch.~G.~L\"oflund,
Stuttgart). The above excerpt translates roughly to: \itshape
as much praise as is given to the spheres is due to the bald head too\upshape.}
\end{flushright}

\smallskip

An isometric action of a group $G$ on a \cat space $X$ is said to be \textbf{minimal}\index{minimal!action} if
there is no non-empty $G$-invariant closed convex subset $X'\subsetneq X$; the space $X$  is itself called
\textbf{minimal}\index{minimal!\cat space} if its full isometry group acts minimally. A \cat space $X$ is called
\textbf{boundary-minimal}\index{boundary-minimal|see{minimal}}\index{minimal!boundary-} if it possesses no
closed convex subset $Y \subsetneq X$ such that $\bd Y = \bd X$. Here is how these notions relate to one
another.

\begin{prop}\label{prop:boundary-minimal}
Let $X$ be a proper \cat space.
\begin{enumerate}
\item Assume $\bd X$ finite-dimensional. If $X$ is minimal, then it is boundary-minimal.

\item Assume $\Isom(X)$ has full limit set. If $X$ is boundary-minimal, then it is minimal.

\item If $X$ is cocompact and geodesically complete, then it is both minimal and boundary-minimal.
\end{enumerate}
\end{prop}

We emphasize that it is \emph{not} true in general that a minimal \cat space is geodesically complete, even if
one assumes that the isometry group acts cocompactly and without global fixed point at infinity.

\sepa

%================================================================================
\noindent\textbf{Group decompositions.} We now turn to structure results on the locally compact isometry group
$\Isom(X)$ of a proper \cat space $X$ in a broad generality; we shall mostly ask that no point at infinity be
fixed simultaneously by all isometries of $X$. This non-degeneracy assumption will be shown to hold
automatically in the presence of lattices (see Theorem~\dref{thm:Lattice=>NoFixedPointAtInfty}).

The the following result was the starting point of this work.

\begin{thm}\label{thm:Decomposition}
Let $X$ be a proper \cat space with finite-dimensional Tits boundary. Assume that $\Isom(X)$ has no global fixed
point in $\bd X$.

Then there is a canonical closed convex $\Isom(X)$-stable subset $X' \se X$ such that $G= \Isom(X')$ has a
finite index open characteristic subgroup $G^*\lhd G$ which admits a canonical decomposition
\begin{equation}\label{eq:GroupDecomposition}
G^*  \cong\ S_1\times \cdots \times S_p \times \big(\RR^n\rtimes \mathbf{O}(n)\big) \times D_1\times \cdots
\times D_q  \kern10mm(p,q,n\geq 0)
\end{equation}
where $S_i$ are almost connected simple Lie groups with trivial centre and $D_j$ are totally disconnected
irreducible groups with trivial amenable radical. Any product decomposition of $G^*$ is a regrouping of the
factors in~\eqref{eq:GroupDecomposition}.

Moreover, all non-trivial normal, subnormal or ascending subgroups $N<D_j$ are still irreducible with trivial
amenable radical and trivial centraliser in $D_j$. (These properties also hold for lattices in $N$ and their
normal, subnormal or ascending subgroups, see~\cite{Caprace-Monod_discrete}.)
\end{thm}

\noindent
(A topological group is called \textbf{irreducible}\index{irreducible!group} if no finite index open
subgroup splits non-trivially as a direct product of closed subgroups. The \textbf{amenable
radical}\index{amenable radical|see{radical}}\index{radical!amenable} of a locally compact group is the largest
amenable normal subgroup; it is indeed a \emph{radical} since the class of amenable locally compact groups is
stable under group extensions.)

\begin{remarks}\
\begin{enumerate}
\item The finite-dimensionality assumption holds automatically when $X$ has a cocompact group of
isometries~\cite[Theorem~C]{Kleiner}. It is also automatic for uniquely geodesic spaces, \emph{e.g.}~manifolds
(Proposition~\ref{prop:dimension:unique:geod}).

\item The conclusion fails in various ways if $G$ fixes a point in $\bd X$.

\item The quotient $G/G^*$ is just a group of permutations of possibly isomorphic factors in the decomposition.
In particular, $G=G^* \rtimes G/G^*$.

\item The canonical continuous homomorphism $\Isom(X) \to \Isom(X') = G$ is proper, but its image sometimes has
infinite covolume.
\end{enumerate}
\end{remarks}

In Theorem~\ref{thm:Decomposition}, the condition that $\Isom(X)$ has no global fixed point at infinity ensures
the existence of a closed convex $\Isom(X)$-invariant subset $Y \se X$ on which $\Isom(X)$ acts minimally (see
Proposition~\ref{prop:EasyDichotomy}). The set of these minimal convex subsets possesses a canonical element,
which is precisely the space $X'$ which appears in Theorem~\ref{thm:Decomposition}.
Proposition~\ref{prop:boundary-minimal} explains why the distinction between $X$ and $X'$ did not appear in
Theorem~\ref{thm:GeodesicallyComplete}.

\sepa

%================================================================================
\noindent\textbf{De Rham decompositions.}
It is known that product decompositions of isometry groups acting
minimally and without global fixed point at infinity induce a splitting of the space (for cocompact Hadamard
manifolds, this is the Lawson--Yau~\cite{Lawson-Yau} and Gromoll--Wolf~\cite{Gromoll-Wolf} theorem; in general
and for more references, see~\cite{Monod_superrigid}). It is therefore natural that
Theorem~\ref{thm:Decomposition} is supplemented by a geometric statement.

\begin{addendum}\label{addendum}
In Theorem~\ref{thm:Decomposition}, there is a canonical isometric decomposition
\begin{equation}\label{eq:SpaceDecomposition}
X' \cong\ X_1\times \cdots \times X_p \times \RR^n \times Y_1\times \cdots \times Y_q
\end{equation}
where $G^*$ acts componentwise according to~\eqref{eq:GroupDecomposition} and $G/G^*$ permutes any isometric
factors. All $X_i$ and $Y_j$ are irreducible and minimal.
\end{addendum}

As it turns out, a geometric decomposition is the first of two independent steps in the proof of
Theorem~\ref{thm:Decomposition}. In fact, we begin with an analogue of the de Rham decomposition~\cite{deRham52}
whose proof uses (a modification of) arguments from the generalised de Rham theorem of
Foertsch--Lytchak~\cite{FoertschLytchak06}. In purely geometrical terms, we have the following statement.

\begin{thm}\label{thm:deRham}
Let $X$ be a proper boundary-minimal \cat space with $\bd X$ finite-dimensional. Then $X$ admits a canonical maximal
isometric splitting
\begin{equation*}
X \cong\ \RR^n \times X_1\times \cdots \times X_m \kern10mm(n,m\geq 0)
\end{equation*}
with each $X_i$ irreducible and~$\neq \RR^0, \RR^1$. Every isometry of $X$ preserves this decomposition upon
permuting possibly isometric factors $X_i$. Moreover, if $X$ is minimal, so is every $X_i$.
\end{thm}

To apply this theorem, it is desirable to know conditions ensuring boundary-minimality. In addition to the
conditions provided by Proposition~\ref{prop:boundary-minimal}, we show that a canonical boundary-minimal
subspace exists as soon as the boundary has circumradius~$>\pi/2$ (Corollary~\ref{cor:CanonicalFullSubset}).

\medskip

In the second part of the proof of Theorem~\ref{thm:Decomposition},
we analyse the \textbf{irreducible} case where $X$ admits no
isometric splitting, resulting in Theorem~\ref{thm:geometric_simplicity} to which we shall now turn.
Combining these two steps, we then prove the unique
decomposition of the \emph{groups}, using also the splitting theorem from~\cite{Monod_superrigid}.

\sepa

%================================================================================
\noindent\textbf{Geometry of normal subgroups.}
In É.~Cartan's correspondence between symmetric spaces and
semi-simple Lie groups as well as in Bruhat--Tits theory, irreducible factors of the space correspond to
\emph{simple} groups. For general \cat spaces and groups, simplicity fails of course very dramatically (free
groups are perhaps the simplest, and yet most non-simple, \cat groups). Nonetheless, we establish a geometric
weakening of simplicity.

\begin{thm}\label{thm:geometric_simplicity}
Let $X\neq \RR$ be an irreducible proper \cat space with finite-dimensional Tits boundary and $G<\Isom(X)$ any
subgroup whose action is minimal and does not have a global fixed point in $\bd X$.

Then every non-trivial normal subgroup $N\lhd G$ still acts minimally and without fixed point in $\bd X$.
Moreover, the amenable radical of $N$ and the centraliser $\centra_{\Isom(G)}(N)$ are both trivial; $N$ does not
split as a product.
\end{thm}

This result can for instance be combined with the solution to Hilbert's fifth problem in order to understand
the connected component of the isometry group.

\begin{cor}\label{cor:Isom(irreducible)}
$\Isom(X)$ is either totally disconnected or an almost connected simple Lie group with trivial centre.

The same holds for any closed subgroup acting minimally and without fixed point at infinity.
\end{cor}

A more elementary application of Theorem~\ref{thm:geometric_simplicity} uses the fact that elements with a
discrete conjugacy class have open centraliser.

\begin{cor}\label{cor:NonDiscrete}
If $G$ is non-discrete, $N$ cannot be a finitely generated discrete subgroup.
\end{cor}

A feature of Theorem~\ref{thm:geometric_simplicity} is that is can be iterated and thus applies to subnormal
subgroups. Recall that more generally a subgroup $H<G$ is \textbf{ascending}\index{ascending} if there is a
(possibly transfinite) chain of normal subgroups starting with $H$ and abutting to $G$. Using limiting
arguments, we bootstrap Theorem~\ref{thm:geometric_simplicity} and show:

\begin{thm}\label{thm:bootstrap}
Let $N<G$ be any non-trivial ascending subgroup. Then all conclusions of Theorem~\ref{thm:geometric_simplicity}
hold for $N$.
\end{thm}

\sepa

%================================================================================
\noindent\textbf{A few cases of superrigidity.} Combining the preceding general structure results with some of
Margulis' theorems, we obtain the following superrigidity statement.

\begin{thm}\label{thm:superrigidity}
Let $X$ be a proper \cat space whose isometry group acts cocompactly and without global fixed point at infinity.
Let $\Gamma = \SL_n(\ZZ)$ with $n \geq 3$ and $G = \SL_n(\RR)$.

For any isometric $\Gamma$-action on $X$ there is a non-empty $\Gamma$-invariant closed convex subset $Y\se X$
on which the $\Gamma$-action extends uniquely to a continuous isometric action of $G$.

\smallskip\nobreak\noindent
$($The corresponding statement applies to all those lattices in semi-simple Lie groups that have virtually
bounded generation by unipotents.$)$
\end{thm}

Observe that the above theorem has no assumptions whatsoever on the action; cocompactness is an assumption on
the given \cat space. It can happen that $\Gamma$ fixes points in $\bd X$, but its action on $Y$ is without
fixed points at infinity and minimal (as we shall establish in the proof).

The assumption on bounded generation holds conjecturally for all non-uniform irreducible lattices in higher rank
semi-simple Lie groups (but always fails in rank one). It is known to hold for arithmetic groups in split or
quasi-split algebraic groups of a number field $K$ of $K$-rank~$\geq 2$ by~\cite{Tavgen}, as well as in a few
cases of isotropic but non-quasi-split groups~\cite{ErovenkoRapinchuk}; see also~\cite{Witte07}.

\smallskip

More generally, Theorem~\ref{thm:superrigidity} holds for (S-)arithmetic groups provided the arithmetic subgroup
(given by integers at infinite places) satisfies the above bounded generation property. For instance, the
$\SL_n$ example is as follows:

\begin{thm}\label{thm:superrigidity_bis}
Let $X$ be a proper \cat space whose isometry group acts cocompactly and without global fixed point at infinity.
Let $m$ be an integer with distinct prime factors $p_1, \ldots p_k$ and set
$$\Gamma = \SL_n(\ZZ[{\scriptstyle\frac{1}{m}}]), \kern5mm
G = \SL_n(\RR) \times \SL_n(\QQ_{p_1}) \times \cdots \times  \SL_n(\QQ_{p_k}),$$
where $n\geq 3$. Then for any isometric $\Gamma$-action on $X$ there is a non-empty $\Gamma$-invariant closed
convex subset $Y\se X$ on which the $\Gamma$-action extends uniquely to a continuous isometric action of $G$.
\end{thm}

We point out that a fixed point property for similar groups acting on low-dimensional \cat cell complexes
was established by B.~Farb~\cite{Farb:Helly}.

\medskip

Some of our general results also allow us to improve on the generality of the \cat superrigidity theorem for irreducible
lattices in arbitrary products of locally compact groups proved in~\cite{Monod_superrigid}. For actions on proper \cat
spaces, the results of \emph{loc.\ cit.} establish an unrestricted superrigidity \emph{on the boundary} but require,
in order to deduce superrigidity on the space itself, the assumption that the action be \emph{reduced} (or
alternatively ``indecomposable'').

We prove that, as soon as the boundary is finite-dimensional,
any action without global fixed point at infinity is always reduced after suitably passing to subspaces and direct factors.
It follows that the superrigidity theorem for arbitrary products holds in that generality,
see Theorem~\ref{thm:Monod:superrigidity} below.

\newpage\tableofcontents\newpage

%================================================================================
\section{Notation and preliminaries}\label{sec:notation}
A metric space is \textbf{proper}\index{proper space} if every closed
ball is compact.

We refer to Bridson and Haefliger~\cite{Bridson-Haefliger} for background on \cat spaces. We recall that the
\textbf{comparison angle}\index{angle!comparison} $\cangle pxy$ determined by three points $p,x,y$ in any metric
space is defined purely in terms of the corresponding three distances by looking at the corresponding Euclidean
triangle. In other words, it is defined by
$$d^2(x,y) = d^2(p,x) + d^2(p,y)- 2 d(p,y)d(p,y)\cos\cangle pxy.$$
The \textbf{Alexandrov angle}\index{angle!Alexandrov} $\aangle pxy$ in a \cat space $X$ is the non-increasing
limit of the comparison angle near $p$ along the geodesic segments $[p,x]$ and $[p,y]$,
see~\cite[II.3.1]{Bridson-Haefliger}. In particular, $\aangle pxy \leq \cangle pxy$. Likewise, geodesic rays
from $p$ determine the Alexandrov angle $\aangle p\xi\eta$ for $\xi, \eta\in\bd X$. The \textbf{Tits
angle}\index{angle!Tits}\index{Tits!angle|see{angle}} $\tangle\xi\eta$ is defined as the supremum of $\aangle p\xi\eta$ over all $p\in X$ and
has several useful characterisations given in Proposition~II.9.8 of~\cite{Bridson-Haefliger}.

\medskip

Recall that to any point at infinity $\xi \in \bd X$ is associated the \textbf{Busemann function}\index{Busemann!function}
$$B_\xi : X \times X \to \RR : (x, y) \mapsto B_{\xi, x}(y)$$
defined by $B_{\xi, x}(y) = \lim_{t \to\infty} (d(\ro(t), y) - d(\ro(t),x))$, where $\ro : [0, \infty) \to X$
is any geodesic ray pointing towards $\xi$. The Busemann function does not depend on the choice of $\ro$ and
satisfies the following:
$$\begin{array}{rcl}
B_{\xi, x}(y) &=& - B_{\xi, y}(x)\\
B_{\xi, x}(z) & = & B_{\xi, x}(y) + B_{\xi, y}(z) \kern10mm\text{(the ``cocycle relation'')}\\
B_{\xi, x}(y) &\leq& d(x, y).
\end{array}$$
Combining the definition of the Busemann function and of the comparison angle, we find that if $r$ is the geodesic ray
pointing towards $\xi$ with $r(0)=x$, then for any $y\neq x$ we have
$$\lim_{t\to\infty}\cos\cangle x{r(t)}y = -\frac{B_{\xi, x}(y)}{d(x,y)} \kern10mm\text{(the ``asymptotic angle formula'')}.%
\index{angle!asymptotic formula}$$
By abuse of language, one refers to \textbf{a} Busemann function when it is more convenient to consider the convex
$1$-Lipschitz function $b_\xi: X\to \RR$ defined by $B_{\xi, x}$ for some (usually implicit) choice of base-point $x\in X$.
We shall simply denote such a function by $b_\xi$ in lower case; they all differ by a constant only in view of the cocycle relation.

\medskip

The boundary at infinity $\bd X$ is endowed with the cône topology~\cite[II.8.6]{Bridson-Haefliger}\index{cone@cône!topology}
as well as with the (much finer) topology defined by the Tits
angle. The former is often implicitly understood, but when referring to dimension or radius, the topology and distance
defined by the Tits angle are considered (this is sometimes emphasised by referring to the ``Tits boundary'').\index{Tits!boundary}
The later distance is not to be confused with the associated length metric called ``Tits distance''\index{Tits!distance}
in the literature; we will not need this concept (except in the discussions at the beginning of Section~\ref{sect:cocompact}).

\medskip

Recall that any complete \cat space splits off a canonical maximal Hilbertian factor%
\index{factor!(pseudo-)Euclidean|see{Euclidean}}\index{Euclidean!factor}
(Euclidean in the proper case studied here) and any isometry decomposes accordingly, see Theorem~II.6.15(6) in~\cite{Bridson-Haefliger}.

\medskip

Normalisers and centralisers in a group $G$ are respectively denoted by $\norma_G$ and $\centra_G$.
When some group $G$ acts on a set and $x$ is a member of this set, the stabiliser of $x$ in $G$
is denoted by $\Stab_G(x)$ or by the shorthand $G_x$.
For the notation regarding algebraic groups, we follow the standard notation as in~\cite{Margulis}.

\bigskip

Finally, we present two remarks that will never be used below but give some context on certain frequent assumptions.

\medskip
The first remark is the following \cat version of the Hopf--Rinow theorem: \emph{Every geodesically complete
locally compact \cat space is proper}. Surprisingly, we could not find this statement in the literature (though
a different statement is often referred to as the Hopf--Rinow theorem, see~\cite[I.3.7]{Bridson-Haefliger}). As
pointed out orally by A.~Lytchak, the above result is readily established by following the strategy of proof
of~\cite[I.3.7]{Bridson-Haefliger} and extending geodesics.

\medskip
The second fact is that if a proper \cat space is finite-dimensional (in the sense of~\cite{Kleiner}), then so
is its Tits boundary (generalising for instance Proposition~\ref{prop:dimension:unique:geod} below). The
argument is given in~\cite[Proposition~2.1]{Caprace-Lytchak} and may be outlined as follows. For any sphere $S$
in the space $X$, the ``visual map'' $\bd X\to S$ is Tits-continuous; if it were injective, the result would
follow. However, it becomes injective after replacing $S$ with the ultraproduct of spheres of unbounded radius
by the very definition of the boundary; the ultraproduct construction preserves the bound on the dimension,
finishing the proof.

%\newpage

%%%%%%%%%%%%%%%%%%%%%%%%%%%%%%%%%%%%%%%%%%%%%%%%%%%%%%%%%%%%%%%%%%%%%%%%%%%%%%
%%%%%%%%%%%%%%%%%%%%%%%%%%%%%%%%%%%%%%%%%%%%%%%%%%%%%%%%%%%%%%%%%%%%%%%%%%%%%%
%\part{Structure theory of the isometry group}
%%%%%%%%%%%%%%%%%%%%%%%%%%%%%%%%%%%%%%%%%%%%%%%%%%%%%%%%%%%%%%%%%%%%%%%%%%%%%%
%%%%%%%%%%%%%%%%%%%%%%%%%%%%%%%%%%%%%%%%%%%%%%%%%%%%%%%%%%%%%%%%%%%%%%%%%%%%%%
\section{Convex subsets of the Tits boundary}\label{sec:TitsBd}
%%%%%%%%%%%%%%%%%%%%%%%%%%%%%%%%%%%%%%%%%%%%%%%%%%%%%%%%%%%%%%%%%%%%%%%%%%%%%%

\subsection{Boundary subsets of small radius}\label{sec:SmallRadius}
Given a metric space $X$ and a subset $Z \subseteq X$, one defines the \textbf{circumradius}\index{circumradius}
of $Z$ in $X$ as
$$\inf_{x \in X}\, \sup_{z \in Z}\, d(x,z).$$
A point $x$ realising the infimum is called a \textbf{circumcentre}\index{circumcentre} of $Z$ in $X$. The
\textbf{intrinsic circumradius}\index{circumradius!intrinsic} of $Z$ is its circumradius in $Z$ itself; one
defines similarly an \textbf{intrinsic circumcentre}\index{circumcentre!intrinsic}. It is called
\textbf{canonical}\index{circumcentre!canonical} if it is fixed by every isometry of $X$ which stabilises $Z$.
We shall make frequent use of the following construction of circumcentres, due to A.~Balser and
A.~Lytchak~\cite[Proposition~1.4]{BalserLytchak_Centers}:

\begin{prop}\label{prop:BalserLytchak}
Let $X$ be a complete CAT(1) space and $Y \subseteq X$ be a finite-dimensional closed convex subset. If $Y$ has
intrinsic circumradius~$\leq \pi/2$, then the set $C(Y)$ of intrinsic circumcentres of $Y$ has a unique
circumcentre, which is therefore a canonical (intrinsic) circumcentre of $Y$.\hfill\qedsymbol
\end{prop}

Let now $X$ be a proper \cat space.

\begin{prop}\label{prop:NestedSequence}
Let $X_0 \supset X_1 \supset \dots$ be a nested sequence of non-empty closed convex subsets of $X$ such that
$\bigcap_n X_n$ is empty. Then the intersection $\bigcap_n \bd X_n$ is a non-empty closed convex subset of $\bd
X$ of intrinsic circumradius at most~$\pi/2$.

In particular, if the Tits boundary is finite-dimensional, then $\bigcap_n \bd X_n$ has a canonical intrinsic
circumcentre.
\end{prop}

\begin{proof}
Pick any $x\in X$ and let $x_n$ be its projection to $X_n$. The assumption $\bigcap_n X_n=\varnothing$ implies
that $x_n$ goes to infinity. Upon extracting, we can assume that it converges to some point $\xi\in\bd X$;
observe that $\xi\in\bigcap_n \bd X_n$. We claim that any $\eta\in \bigcap_n \bd X_n$ satisfies $\tangle\xi\eta
\leq \pi/2$. The proposition then follows because (i)~the boundary of any closed convex set is closed and
$\pi$-convex~\cite[II.9.13]{Bridson-Haefliger} and (ii)~each $\bd X_n$ is non-empty since otherwise $X_n$ would
be bounded, contradicting $\bigcap_n X_n=\varnothing$. When $\bd X$ has finite dimension, there is a canonical
intrinsic circumcentre by Proposition~\ref{prop:BalserLytchak}.

For the claim, observe that there exists a sequence of points $y_n\in X_n$ converging to $\eta$. It suffices to
prove that the comparison angle $\cangle x{x_n}{y_n}$ is bounded by~$\pi/2$ for all $n$,
see~\cite[II.9.16]{Bridson-Haefliger}. This follows from
$$\cangle{x_n}x{y_n} \geq \aangle{x_n}x{y_n} \geq\pi/2,$$
where the second inequality holds by the properties of the projection on a convex
set~\cite[II.2.4(3)]{Bridson-Haefliger}.
\end{proof}

The combination of the preceding two propositions has the following consequence, which improves the results
established by Fujiwara, Nagano and Shioya (Theorems~1.1 and~1.3 in~\cite{Fujiwara_et_al}).

\begin{cor}\label{cor:FixedPointParabolic}
Let $g$ be a parabolic isometry of $X$. The following assertions hold:
\begin{enumerate}
\item The fixed point set of $g$ in $\bd X$ has intrinsic circumradius at most~$\pi/2$.

\item If $\bd X$ finite-dimensional, then the centraliser $\centra_{\Isom(X)}(g)$ has a canonical global
fixed point in $\bd X$.\label{pt:FixedPointParabolic:centra}

\item For any subgroup $H < \Isom(X)$ containing $g$, the (possibly empty) fixed point set of $H$ in $\bd
X$ has circumradius at most~$\pi/2$.\hfill\qed
\end{enumerate}
\end{cor}

Here is another immediate consequence.

\begin{cor}\label{cor:IncreasingCompact}
Let $G$ be a topological group with a continuous action by isometries on $X$ without global fixed point. Suppose
that $G$ is the union of an increasing sequence of compact subgroups and that $\bd X$ is finite-dimensional.
Then there is a canonical $G$-fixed point in $\bd X$, fixed by all isometries normalising $G$.
\end{cor}

\begin{proof}
Consider the sequence of fixed point sets $X^{K_n}$ of the compact subgroups $K_n$. Its intersection is empty by
assumption and thus Proposition~\ref{prop:NestedSequence} applies.
\end{proof}

Finally, we record the following elementary fact, which may also be deduced by means of
Proposition~\ref{prop:NestedSequence}:

\begin{lem}\label{lem:RadiusHoroball}
Let $\xi \in \bd X$. Given any closed horoball $B$ centred at $\xi$, the boundary $\bd B$ coincides with the
ball of Tits radius~$\pi/2$ centred at $\xi$ in $\bd X$.
\end{lem}

\begin{proof}
Any two horoballs centred at the same point at infinity lie at bounded Hausdorff distance from one another.
Therefore, they have the same boundary at infinity. In particular, the boundary $\bd B$ of the given horoball
coincides with the intersection of the boundaries of all horoballs centred at $\xi$. By
Proposition~\ref{prop:NestedSequence}, this is of circumradius at most~$\pi/2$; in fact the proof of that proposition shows precisely that
the set is contained in the ball of radius at most~$\pi/2$ around $\xi$.

Conversely, let $\eta \in \bd X$ be a point which does not belong to $\bd B$. We claim that
$\tangle{\xi}{\eta} \geq \pi/2$. This shows that every point of $\bd X$ at Tits distance less
than~$\pi/2$ from $\xi$ belongs to $\bd B$. Since the latter is closed, it follows that $\bd B$ contains
the closed ball of Tits radius~$\pi/2$

We turn to the claim. Let $b_\xi$ be a Busemann function centred at $\xi$.
Since every geodesic ray pointing towards $\eta$ escapes every horoball centred at $\xi$,
there exists a ray $\ro: [0, \infty) \to X$ pointing to $\eta$ such that  $b_\xi(\ro(0)) = 0$ and
$b_\xi(\ro(t)) > 0$ for all $t > 0$ (actually, this increases to infinity by convexity).
Let $c: [0, \infty) \to X$ be the geodesic ray emanating from $\ro(0)$
and pointing to $\xi$. We have $\tangle{\xi}{\eta} = \lim_{t,s\to\infty}\cangle{\ro(0)}{\ro(t)}{c(s)}$,
see~\cite[II.9.8]{Bridson-Haefliger}. Therefore the claim follows from the asymptotic angle formula
(Section~\ref{sec:notation}) by taking $y=c(s)$ with $s$ large enough.
\end{proof}

%================================================================================
\subsection{Subspaces with boundary of large radius}
\label{sec:largeRadius}
As before, let $X$ be a proper \cat space. The following result improves Proposition~2.2 in~\cite{Leeb}:

\begin{prop}\label{prop:ConvexCore}
Let $Y \subseteq X$ be a closed convex subset such that $\bd Y$ has intrinsic circumradius~$>\pi/2$. Then there
exists a closed convex subset $Z\se X$ with $\bd Z= \bd Y$ which is minimal for these properties. Moreover, the
union $Z_0$ of all such minimal subspaces is closed, convex and splits as a product $Z_0 \cong Z \times Z'$.
\end{prop}

\begin{proof}
If no minimal such $Z$ existed, there would be a chain of such subsets with empty intersection. The distance to
a base-point must then go to infinity and thus the chain contains a countable sequence to which we apply
Proposition~\ref{prop:NestedSequence}, contradicting the assumption on the circumradius.

Let $Z'$ denote the set of all such minimal sets and $Z_0 = \bigcup Z'$ be its union. As in~\cite[p.~10]{Leeb}
one observes that for any $Z_1, Z_2 \in Z'$, the distance $z \mapsto d(z, Z_2)$ is constant on $Z_1$ and that
the nearest point projection $p_{Z_2}$ restricted to $Z_1$ defines an isometry $Z_1 \to Z_2$. By the Sandwich
Lemma~\cite[II.2.12]{Bridson-Haefliger}, this implies that $Z_0$ is convex and that the map $Z' \times Z' \to
\RR_+ : (Z_1, Z_2) \mapsto d(Z_1, Z_2)$ is a geodesic metric on $Z'$. As in~\cite[Section~4.3]{Monod_superrigid},
this yields a bijection $\alpha : Z_0 \to Z \times Z': x \mapsto (p_Z(x), Z_x)$, where $Z_x$ is the unique
element of $Z'$ containing $x$. The product of metric spaces $Z \times Z'$ is given the product metric. In order
to establish that $\alpha$ is an isometry, it remains as in~\cite[Proposition~38]{Monod_superrigid}, to trivialise
``holonomy''; it the current setting, this is achieved by Lemma~\ref{lem:holonomy}, which thus concludes the
proof of Proposition~\ref{prop:ConvexCore}. (Notice that $Z_0$ is indeed closed since otherwise we could extend
$\alpha^{-1}$ to the completion of $Z\times Z'$.)
\end{proof}

\begin{lem}\label{lem:holonomy}
For all $Z_1, Z_2,Z_3 \in Z'$, we have $p_{Z_1} \circ p_{Z_3} \circ p_{Z_2} |_{Z_1} = \mathrm{Id}_{Z_1}$.
\end{lem}
\begin{proof}[Proof of Lemma~\ref{lem:holonomy}]
Let $\vartheta : Z_1 \to Z_1$ be the isometry defined by $p_{Z_1} \circ p_{Z_3} \circ p_{Z_2} |_{Z_1}$ and let
$f$ be its displacement function. Then $f : Z_1 \to \RR$ is a non-negative convex function which is bounded
above by $d(Z_1, Z_2)+d(Z_2, Z_3) + d(Z_3, Z_1)$. In particular, the restriction of $f$ to any geodesic ray in
$Z_1$ is non-increasing. Therefore, a sublevel set of $f$ is a closed convex subset $Z$ of $Z_1$ with full
boundary, namely $\bd Z= \bd Z_1$. By definition, the subspace $Z_1$ is minimal with respect to the property
that $\bd Z_1 = \bd Y$ and hence we deduce $Z= Z_1$. It follows that the convex function $f$ is constant. In
other words, the isometry $\vartheta$ is a Clifford translation. If it is not trivial, then $Z_1$ would contain
a $\teta$-stable geodesic line on which $\teta$ acts by translation. But by
\cite[Lemma~II.2.15]{Bridson-Haefliger}, the restriction of $\vartheta$ to any geodesic line is the identity.
Therefore $\vartheta$ is trivial, as desired.
\end{proof}

Let $\Gamma$ be a group acting on $X$ by isometries. Following~\cite[Definition~5]{Monod_superrigid}, we say
that the $\Gamma$-action is \textbf{reduced}\index{reduced action} if there is no unbounded closed convex subset
$Y \subsetneq X$ such that $g.Y$ is at finite Hausdorff distance from $Y$ for all $g \in \Gamma$.

\begin{cor}\label{cor:ReducedAction}
Let $X$ be a proper irreducible \cat space with finite-dimensional Tits boundary, and $\Gamma < \Isom(X)$ be a
subgroup acting minimally without fixed point at infinity. Then the $\Gamma$-action is reduced.
\end{cor}

\begin{proof}
Suppose for a contradiction that the $\Gamma$-action on $X$ is not reduced. Then there exists an unbounded
closed convex subset $Y \subsetneq X$ such that $g.Y$ is at finite Hausdorff distance from $Y$ for all $g \in
\Gamma$. In particular $\bd Y$ is $\Gamma$-invariant. By Proposition~\ref{prop:BalserLytchak}, it must have
intrinsic circumradius $> \pi/2$. Proposition~\ref{prop:ConvexCore} therefore yields a canonical closed convex
subset $Z_0 = Z \times Z'$ with $\bd (Z \times \{z'\}) = \bd Y$ for all $z' \in Z'$; clearly $Z_0$ is
$\Gamma$-invariant and hence we have $Z_0 = X$ by minimality. Since $X$ is irreducible by assumption, we deduce
$X = Z$ and hence $X = Y$, as desired.
\end{proof}

%================================================================================
\subsection{Minimal actions and boundary-minimal spaces}
Boundary-minimality and minimality, as defined in the Introduction, are two possible ways for a
\cat space to be ``non-degenerate'', as illustrated by the following.

\begin{lem}\label{lem:ConvexFunctions}
Let $X$ be a complete \cat space.
\begin{enumerate}
\item A group $G < \Isom(X)$ acts minimally if and only if any continuous convex $G$-invariant function on
$X$ is constant.

\item If $X$ is boundary-minimal then any bounded convex function on  $X$ is constant.
\end{enumerate}
\end{lem}

\begin{proof}
Necessity in the first assertion follows immediately by considering sub-level sets (see~\cite[Lemma~37]{Monod_superrigid}).
Sufficiency is due to the fact that the distance to a closed convex set is a convex continuous
function~\cite[II.2.5]{Bridson-Haefliger}.
The second assertion was established in the proof of Lemma~\ref{lem:holonomy}.
\end{proof}

Proposition~\ref{prop:ConvexCore} has the following important consequence:

\begin{cor}\label{cor:CanonicalFullSubset}
Let $X$ be a proper \cat space. If $\bd X$ has circumradius $> \pi/2$, then $X$ possesses a canonical closed
convex subspace $Y \subseteq X$ such that $Y$ is boundary-minimal and $\bd Y = \bd X$.
\end{cor}

\begin{proof}
Let $Z_0 = Z \times Z'$ be the product decomposition provided by Proposition~\ref{prop:ConvexCore}. The group
$\Isom(X)$ permutes the elements of $Z'$ and hence acts by isometries on $Z'$. Under the present hypotheses, the
space $Z'$ is bounded since $\bd Z = \bd X$. Therefore it has a circumcentre $z'$, and the fibre $Y = Z \times
\{z'\}$ is thus $\Isom(X)$-invariant.
\end{proof}

\begin{prop}\label{prop:LargeCircumradius}
Let $X$ be a proper \cat space which is minimal. Assume that $\bd X$ has finite dimension. Then $\bd X$ has
circumradius~$>\pi/2$ (unless $X$ is reduced to a point). In particular, $X$ is boundary-minimal.
\end{prop}

The proof of Proposition~\ref{prop:LargeCircumradius} requires some preliminaries.
Given a point at infinity $\xi$, consider the Busemann function $B_\xi$; the cocycle property (recalled in
Section~\ref{sec:notation}) implies in particular that for any isometry $g \in \Isom(X)$ fixing $\xi$ and any $x \in X$
the real number $B_{\xi, x}(g.x)$ is independent on the choice of $x$ and yields a canonical homomorphism
$$
\beta_\xi : \Isom(X)_\xi \lra \RR : g \longmapsto B_{\xi, x}(g.x)
$$
called the \textbf{Busemann character}\index{Busemann!character} centred at $\xi$.

\smallskip

Given an isometry $g$, it follows by the \cat property that $\inf_{n\geq 0} d(g^n x, x)/n$ coincides with the translation length
of $g$ independently of $x$. We call an isometry \textbf{ballistic}\index{isometry!ballistic} when this number is positive.
An important fact about a ballistic isometry $g$ of any complete \cat space $X$ is that
for any $x \in X$ the sequence $\{g^n.x\}_{n\geq 0}$ converges to a point $\eta_g \in
\bd X$ independent of $x$; $\eta_g$ is called the (canonical) \textbf{attracting fixed point} of $g$ in $\bd X$.
Moreover, this convergence holds also in angle, which means that $\lim \cangle{x}{g^n x}{r(t)}$ vanishes as $n, t\to\infty$
when $r:\RR_+\to X$ is any ray pointing to $\eta_g$.
This is a (very) special case of the results in~\cite{KarlssonMargulis}.

\begin{lem}\label{lem:Busemann:character}
Let $\xi \in X$ and $g \in \Isom(X)_\xi$ be an isometry which is not annihilated by the Busemann character
centred at $\xi$. Then $g$ is ballistic. Furthermore, if $\beta_\xi(g) > 0$ then $\tangle{\xi}{\eta_g} >\pi/2$.
\end{lem}

\begin{proof}
We have $\beta_\xi(g) = B_{\xi, x}(g.x) \leq d(x, g.x)$ for all $x \in X$. Thus $g$ is ballistic as soon as
$\beta_\xi(g)$ is non-zero.

Assume $\beta_\xi(g)>0$ and suppose for a contradiction that $\tangle{\xi}{\eta_g} \leq \pi/2$.
Choose $x\in X$ and let $\ro, \sigma$ be the rays issuing from $x$ and pointing towards $\xi$ and $\eta_g$ respectively.
Recall from~\cite[II.9.8]{Bridson-Haefliger} that $\tangle{\xi}{\eta_g} = \lim_{t,s\to\infty}\cangle x{\ro(t)}{\sigma(s)}$.
The convergence in direction of $g^n x$ implies that this angle is also given by $\lim_{t,n\to\infty}\cangle x{\ro(t)}{g^n x}$.
Since $\beta_\xi(g) > 0$ we can fix $n$ large enough to have
$$\cos\liminf_{t\to\infty}\cangle x{\ro(t)}{g^n x} > -\frac{\beta_\xi(g)}{d(g x, x)}.$$
We now apply the asymptotic angle formula from Section~\ref{sec:notation} with $y=g^n x$ and deduce that the left hand side
is $-\beta_\xi(g^n x) / d(g^n x, x)$. Since $\beta_\xi(g^n x) = n \beta_\xi(g x)$ and $d(g^n x, x)\leq nd(g x, x)$, we have a contradiction.
\end{proof}

\begin{proof}[Proof of Proposition~\ref{prop:LargeCircumradius}]
We can assume that $\bd X$ is non-empty since otherwise $X$ is a point by minimality. Suppose for a
contradiction that its circumradius is~$\leq\pi/2$. Then $\Isom(X)$ possesses a global fixed point $\xi\in\bd X$
and $\xi$ is a circumcentre of $\bd X$, see Proposition~\ref{prop:BalserLytchak}.
Lemma~\ref{lem:Busemann:character} implies that $\Isom(X) = \Isom(X)_\xi$ is annihilated by the Busemann
character centred at $\xi$. Thus $\Isom(X)$ stabilises every horoball, contradicting minimality.
\end{proof}

We shall use repeatedly the following elementary fact.

\begin{lem}\label{lem:cocompact:minimal}
Let $G$ be a group with an isometric action on a proper geodesically complete \cat space $X$.
If $G$ acts cocompactly or more generally has full limit set, then the action is minimal.
(This holds more generally when $\Delta G=\bd X$ in the sense of Section~\ref{sec:dichotomy} below.)
\end{lem}

\begin{proof}
Let $Y\se X$ be a a non-empty closed convex invariant subset, choose $y\in Y$ and suppose for a contradiction
that there is $x\notin Y$. Let $r:\RR_+\to X$ be a geodesic ray starting at $y$ and going through $x$. By
convexity~\cite[II.2.5(1)]{Bridson-Haefliger}, the function $d(r(t), Y)$ tends to infinity and thus
$r(\infty)\notin \bd Y$. This is absurd since $\Delta G\se \bd Y$.
\end{proof}

\begin{proof}[Proof of Proposition~\ref{prop:boundary-minimal}]
(i) See Proposition~\ref{prop:LargeCircumradius}.

\smallskip \noindent (ii) Since $\Isom(X)$ has full limit set, any $\Isom(X)$-invariant subspace has full
boundary. Minimality follows, since boundary-minimality ensures that $X$ possesses no proper subspace with full boundary.

\smallskip \noindent (iii) $X$ is minimal by Lemma~\ref{lem:cocompact:minimal}, hence boundary-minimal by (i), since any
cocompact space has finite-dimensional boundary by~\cite[Theorem~C]{Kleiner}.
\end{proof}

%%%%%%%%%%%%%%%%%%%%%%%%%%%%%%%%%%%%%%%%%%%%%%%%%%%%%%%%%%%%%%%%%%%%%%%%%%%%%%
\section{Minimal invariant subspaces for subgroups}
%%%%%%%%%%%%%%%%%%%%%%%%%%%%%%%%%%%%%%%%%%%%%%%%%%%%%%%%%%%%%%%%%%%%%%%%%%%%%%

\subsection{Existence of a minimal invariant subspace}
For the record, we recall the following elementary dichotomy; a refinement will be given in
Theorem~\ref{thm:dichotomy} below:

\begin{prop}\label{prop:EasyDichotomy}
Let $G$ be a group acting by isometries on a proper \cat. Then either $G$ has a global fixed point at infinity,
or any filtering family of non-empty closed convex $G$-invariant subsets has non-empty intersection.
\end{prop}
\noindent (Recall that a family of sets is filtering if it is directed by containment~$\supseteq$.)

\begin{proof}
(Remark~36 in~\cite{Monod_superrigid}.) Suppose $\mathscr Y$ is such a family, choose $x\in X$ and let $x_Y$ be
its projection on each $Y\in\mathscr Y$. If the net $\{x_Y\}_{Y\in\mathscr Y}$ is bounded, then
$\bigcap_{Y\in\mathscr Y} Y$ is non-empty. Otherwise it goes to infinity and any accumulation point in $\bd X$
is $G$-fixed in view of $d(g x_Y, x_Y) \leq d(g x, x)$.
\end{proof}

%================================================================================
\subsection{Dichotomy}\label{sec:dichotomy}
Let $G$ be a group acting by isometries on a complete \cat space $X$.

\begin{lem}\label{lem:BoundaryOrbits}
 Given any two $x, y \in X$, the convex closures of
the respective $G$-orbits of $x$ and $y$ in $X$ have the same boundary in $\bd X$.
\end{lem}

\begin{proof}
Let $Y$ be the convex closure of the $G$-orbit of $x$. In particular $Y$ is the minimal closed convex
$G$-invariant subset containing $x$. Given any closed convex $G$-invariant subset $Z$, let $r= d(x, Z)$.
Recall that the tubular closed neighbourhood $\mathcal{N}_r(Z)$ is convex~\cite[II.2.5(1)]{Bridson-Haefliger}.
Since it is also $G$-invariant and contains $x$, the minimality of $Y$ implies $Y \se \mathcal{N}_r(Z)$.
\end{proof}

This yields a canonical closed convex $G$-invariant subset of the boundary $\bd X$, which we denote by $\Delta
G$. It contains the limit set $\Lambda G$ but is sometimes larger.

\bigskip

Combining what we established thus far with the splitting arguments from~\cite{Monod_superrigid}, we obtain a
dichotomy:

\begin{thm}\label{thm:dichotomy}
Let $G$ be a group acting by isometries on a complete \cat space $X$ and $H<G$ any subgroup.

\noindent
If $H$ admits no minimal non-empty closed convex invariant subset and $X$ is proper, then:
\begin{itemize}
\item[(A.i)] $\Delta H$ is a non-empty closed convex subset of $\bd X$ of intrinsic circumradius at most~$\pi/2$.

\item[(A.ii)] If $\bd X$ is finite-dimensional, then the normaliser $\norma_G(H)$ of $H$ in $G$ has a global fixed point in $\bd X$.
\end{itemize}

\noindent
If $H$ admits a minimal non-empty closed convex invariant subset $Y \se X$, then:
\begin{itemize}
\item[(B.i)] The union $Z$ of all such subsets is a closed convex $\norma_G(H)$-invariant subset.

\item[(B.ii)] $Z$ splits $H$-equivariantly and isometrically as a product $Z \simeq Y \times C$, where $C$ is a
complete \cat space which admits a canonical $\norma_G(H)/H$-action by isometries.

\item[(B.iii)] If the $H$-action on $X$ is non-evanescent, then $C$ is bounded and there is a canonical minimal
non-empty closed convex $H$-invariant subset which is $\norma_G(H)$-stable.
\end{itemize}
\end{thm}

\noindent%
(When $X$ is proper, the \emph{non-evanescence} condition of~(iii) simply means that $H$ has no fixed
point in $\bd X$; see~\cite{Monod_superrigid}.)

\begin{proof}
In view of Lemma~\ref{lem:BoundaryOrbits}, the set $\Delta H$ is contained in the boundary of any non-empty closed
convex $H$-invariant set and is $\norma_G(H)$-invariant. Thus the assertions~(A.i) and~(A.ii) follow from
Proposition~\ref{prop:NestedSequence}, noticing that in a proper space $\Delta H$ is non-empty unless $H$ has bounded
orbits, in which case it fixes a point, providing a minimal subspace.
For~(B.i), (B.ii) and~(B.iii), see Remarks~39 in~\cite{Monod_superrigid}.
\end{proof}

%================================================================================
\subsection{Normal subgroups}
\begin{proof}[Proof of Theorem~\ref{thm:geometric_simplicity}]
We adopt the notation and assumptions of the theorem. By~(A.ii), $N$ admits a minimal non-empty closed convex invariant
subset $Y \se X$. This set is unbounded, since otherwise $N$ fixes a point and thus by $G$-minimality $X^N=X$,
hence $N=1$.
Since $X$ is irreducible, points~(B.i) and~(B.ii) show $Y=X$ and thus $N$ acts indeed minimally.

Since the displacement function of any $g\in\centra_G(N)$ is a convex $N$-invariant function, it is constant by
minimality. Hence $g$ is a Clifford translation and must be trivial since otherwise $X$ splits off a Euclidean
factor, see~\cite[II.6.15]{Bridson-Haefliger}.

The derived subgroup $N'=[N,N]$ is also normal in $G$ and therefore acts minimally by the previous discussion,
noticing that $N'$ is non-trivial since otherwise $N\se \centra_G(N)$. If $N$ fixed a point at infinity, $N'$
would preserve all corresponding horoballs, contradicting minimality.

Having established that $N$ acts minimally and without fixed point at infinity, we can apply the splitting
theorem (Corollary~10 in~\cite{Monod_superrigid}) and deduce from the irreducibility of $X$ that $N$ does not split.

Finally, let $R\lhd N$ be the amenable radical and observe that it is normal in $G$. The theorem of
Adams--Ballmann~\cite{AB98} states that $R$ either (i)~fixes a point at infinity or (ii)~preserves a Euclidean
flat in $X$. (Although their result is stated for amenable groups without mentioning any topology, the proof
applies indeed to every topological group that preserves a probability measure whenever it acts continuously on
a compact metrisable space.) If $R$ is non-trivial, we know already from the above discussion that~(i) is
impossible and that $R$ acts minimally; it follows that $X$ is a flat. By irreducibility and since $X\neq\RR$,
this forces $X$ to be a point, contradicting $R\neq 1$.
\end{proof}

Corollary~\ref{cor:Isom(irreducible)} will be proved in Section~\ref{sec:decompositions}.
For Corollary~\ref{cor:NonDiscrete}, it suffices to observe that the centraliser of any element of a discrete
normal subgroup is open. Next, we recall the following definition.

A subgroup $N$ of a group $G$ is \textbf{ascending}\index{ascending} if there is a family of subgroups
$N_\alpha<G$ indexed by the ordinals and such that $N_0=N$, $N_\alpha\lhd N_{\alpha+1}$, $N_\alpha =
\bigcup_{\beta<\alpha} N_\beta$ if $\alpha$ is a limit ordinal and $N_\alpha=G$ for $\alpha$ large enough. The
smallest such ordinal is the \textbf{order}\index{order!of an ascending subgroup}.

\begin{prop}\label{prop:EasyBootstrap}
Consider a group acting minimally by isometries on a proper \cat space. Then any ascending subgroup without
global fixed point at infinity still acts minimally.
\end{prop}

\begin{proof}
We argue by transfinite induction on the order $\teta$ of ascending subgroups $N<G$, the case $\teta=0$ being
trivial. Let $X$ be a space as in the statement. By Proposition~\ref{prop:EasyDichotomy}, each $N_\alpha$ has a
minimal set. If $\teta=\teta'+1$, it follows from~(B.iii) that $N_{\teta'}$ acts minimally and we are done by
induction hypothesis. Assume now that $\teta$ is a limit ordinal. For all $\alpha$, we denote as in~(B.i) by
$Z_\alpha\se X$ the union of all $N_\alpha$-minimal sets. The induction hypothesis implies that for all
$\alpha\leq \beta < \teta$, any $N_\beta$-minimal set is $N_\alpha$-minimal. Thus, if $Z_0= Y_0\times C_0$ is a
splitting as in~(B.ii) with a $N$-minimal set $Y_0$, we have a nested family of decompositions $Z_\alpha =
Y_0\times C_\alpha$ for a nested family of closed convex subspaces $C_\alpha$ of the compact \cat space $C_0$,
indexed by $\alpha<\teta$. Thus, for any $c\in\bigcap_{\alpha<\teta} C_\alpha$, the space $Y_0\times \{c\}$ is
$G$-invariant and hence $Y_0=X$ indeed.
\end{proof}

\begin{remark}
Proposition~\ref{prop:EasyBootstrap} holds more generally for complete \cat spaces if $N$ is non-evanescent.
Indeed Proposition~\ref{prop:EasyDichotomy} hold in that generality (Remark~36 in~\cite{Monod_superrigid}) and
$C$ remains compact in a weaker topology (Theorem~14 in~\cite{Monod_superrigid}).
\end{remark}

\begin{proof}[Proof of Theorem~\ref{thm:bootstrap}]
In view of Theorem~\ref{thm:geometric_simplicity}, it suffices to prove that any non-trivial ascending subgroup
$N<G$ as in that statement still acts minimally and without global fixed point at infinity. We argue by
induction on the order $\teta$ and we can assume that $\teta$ is a limit ordinal by
Theorem~\ref{thm:geometric_simplicity}. Then $\bigcap_{\alpha<\teta} (\bd X)^{N_\alpha}$ is empty and thus by
compactness there is some $\alpha<\teta$ such that $(\bd X)^{N_\alpha}$ is empty. Now $N_\alpha$ acts minimally
on $X$ by Proposition~\ref{prop:EasyBootstrap} and thus we conclude using the induction hypothesis.
\end{proof}

%%%%%%%%%%%%%%%%%%%%%%%%%%%%%%%%%%%%%%%%%%%%%%%%%%%%%%%%%%%%%%%%%%%%%%%%%%%%%%
\section{Algebraic and geometric product decompositions}
%%%%%%%%%%%%%%%%%%%%%%%%%%%%%%%%%%%%%%%%%%%%%%%%%%%%%%%%%%%%%%%%%%%%%%%%%%%%%%

\subsection{Preliminary decomposition of the space}
We shall prepare our spaces by means of a geometric decomposition. For any geodesic metric space with
\emph{finite affine rank}, Foertsch--Lytchak~\cite{FoertschLytchak06} established a canonical decomposition
generalising the classical theorem of de~Rham~\cite{deRham52}. However, such a statement fails to be true for
\cat spaces that are merely proper, due notably to compact factors that can be infinite products. Nevertheless,
using asymptotic \cat geometry and Section~\ref{sec:SmallRadius}, we can adapt the arguments
from~\cite{FoertschLytchak06} and obtain:

\begin{thm}\label{thm:preliminary}
Let $X$ be a proper \cat space with $\bd X$ finite-dimensional and of circumradius~$>\pi/2$. Then there is a
canonical closed convex subset $Z\se X$ with $\bd Z=\bd X$, invariant under all isometries, and admitting a
canonical maximal isometric splitting
\begin{equation}\label{eq:preliminary}
Z \cong\ \RR^n \times Z_1\times \cdots \times Z_m \kern10mm(n,m\geq 0)
\end{equation}
with each $Z_i$ irreducible and~$\neq \RR^0, \RR^1$. Every isometry of $Z$ preserves this decomposition upon
permuting possibly isometric factors $Z_i$.
\end{thm}

\begin{remark}\label{rem:preliminary}
It is well known that in the above situation the splitting~\eqref{eq:preliminary} induces a decomposition
$$\Isom(Z)\ =\ \Isom(\RR^n) \times \Big(\big(\Isom(Z_1) \times\cdots\times \Isom(Z_m)\big) \rtimes F\Big),$$
where $F$ is the permutation group of $\{1, \ldots, d\}$ permuting possible isometric factors amongst the $Y_j$.
Indeed, this follows from the statement that isometries preserve the splitting upon permutation of factors, see
\emph{e.g.} Proposition~I.5.3(4) in~\cite{Bridson-Haefliger}. Of course, this does not \emph{a priori} mean that
we have a unique, nor even canonical, splitting \emph{in the category of groups}; this shall however be
established for Theorem~\ref{thm:Decomposition}.
\end{remark}

The hypotheses of Theorem~\ref{thm:preliminary} are satisfied in some naturally occurring situations:

\begin{cor}\label{cor:preliminary}
Let $X$ be a proper \cat space with finite-dimensional boundary.
\begin{enumerate}
\item If $\Isom(X)$ has no fixed point at infinity, then $X$ possesses a subspace $Z$ satisfying all the
conclusions of Theorem~\ref{thm:preliminary}.

\item If $\Isom(X)$ acts minimally, then $X$ admits a canonical splitting as in~\ref{eq:preliminary}.
\end{enumerate}
\end{cor}
\begin{proof}[Proof of Corollary~\ref{cor:preliminary}]
By Proposition~\ref{prop:BalserLytchak}, if $\Isom(X)$ has no fixed point at infinity, then $\bd X$ has
circumradius~$>\pi/2$. By Proposition~\ref{prop:LargeCircumradius}, the same conclusion holds is $\Isom(X)$ acts
minimally.
\end{proof}

\begin{proof}[Proof of Theorems~\ref{thm:deRham} and~\ref{thm:preliminary}]
For Theorem~\ref{thm:preliminary}, we let $Z\se X$ be the canonical boundary-minimal subset with $\bd Z=\bd X$ provided by
Corollary~\ref{cor:CanonicalFullSubset}; we shall not use the circumradius assumption any more. For Theorem~\ref{thm:deRham}, we let $Z=X$.
The remainder of the argument is common for both statements.

Recalling that in complete generality all isometries preserving the Euclidean factor decomposition~\cite[II.6.15]{Bridson-Haefliger},
we can assume that $Z$ has no Euclidean factor and shall obtain the decomposition~\eqref{eq:preliminary} with $n=0$.

Since $Z$ is minimal amongst closed convex subsets with $\bd Z=\bd X$, it has no non-trivial compact factor. On
the other hand, any proper geodesic metric space admits some maximal product decomposition into non-compact
factors. In conclusion, $Z$ admits some maximal splitting $Z=Z_1\times \cdots \times Z_m$ with each $Z_i$
irreducible and~$\neq \RR^0, \RR^1$. (This can fail in presence of compact factors).

It remains to prove that any other such decomposition $Z=Z'_1\times \cdots \times Z'_{m'}$ coincides with the
first one after possibly permuting the factors (in particular, $m'=m$). We now borrow from the argumentation
in~\cite{FoertschLytchak06}, indicating the steps and the necessary changes. It is assumed that the reader has a
copy of~\cite{FoertschLytchak06} at hand but keeps in mind that our spaces might lack the finite affine rank
condition assumed in that paper. We shall replace the notion of affine subspaces with a large-scale particular
case: a \textbf{cône}\index{cone@cône} shall be any subspace isometric to a closed convex cône in some Euclidean
space. This includes the particular cases of a point, a ray or a full Euclidean space.

Whenever a space $Y$ has some product decomposition and $Y'$ is a factor, write $Y'_y\se Y$ for the
corresponding fibre $Y'_y\cong Y'$ through $y\in Y$. The following is an analogue of Corollary~1.2
in~\cite{FoertschLytchak06}.

\begin{lem}\label{lem:EuclideanDecomposition}
Let $Y$ be a proper \cat space with finite-dimensional boundary and without compact factors. Suppose given two
decompositions $Y=Y_1\times Y_2=S_1\times S_2$ with all four $(Y_i)_y\cap (S_j)_y$ reduced to $\{y\}$ for some
$y\in Y$. Then $Y$ is a Euclidean space.
\end{lem}

\begin{proof}[Proof of Lemma~\ref{lem:EuclideanDecomposition}]
Any $y\in Y$ is contained in a maximal cône based at $y$ since $\bd Y$ has finite dimension; by abuse of
language we call such cônes maximal. The arguments of Sections~3 and~4 in~\cite{FoertschLytchak06} show that any
maximal cône is \emph{rectangular}, which means that it inherit a product structure from any product
decomposition of the ambient \cat space. Specifically, it suffices to observe that the product of two cônes is a
cône and that the projection of a cône along a product decomposition of \cat spaces remains a cône. (In fact,
the ``equality of slopes'' of Section~4.2 in~\cite{FoertschLytchak06}, namely the fact that parallel geodesic
segments in a \cat space have identical slopes in product decompositions, is a general fact for \cat spaces. It
follows from the convexity of the metric, see for instance~\cite[Proposition~49]{Monod_superrigid} for a more general
statement.) The deduction of the statement of Lemma~\ref{lem:EuclideanDecomposition} from the rectangularity of
maximal cônes following~\cite{FoertschLytchak06} is particularly short since all proper \cat Banach spaces are
Euclidean.
\end{proof}

\begin{lem}\label{lem:FL_Dichotomy}
For a given $z\in Z$ and any product decomposition $Z=S\times S'$, the intersection $(Z_i)_z \cap S_z$ is either
$\{z\}$ or $(Z_i)_z$.
\end{lem}

\begin{proof}[Proof of Lemma~\ref{lem:FL_Dichotomy}]
Write $P^S: Z\to S$ and $P^{Z_i}: Z\to Z_i$ for the projections and set $F_z=S_z\cap (Z_i)_z$.
Following~\cite{FoertschLytchak06}, define $T\se Z$ by $T=P^S(F_z) \times S'$. We contend that $P^{Z_i}(T)$ has
full boundary in $Z_i$.

Indeed, given any point in $\bd (Z_i)_z$, we represent is by a ray $r$ originating from $z$. We can choose a
maximal cône in $Z$ based at $z$ and containing $r$. We know already that this cône is rectangular, and
therefore the proof of Lemma~5.2 in~\cite{FoertschLytchak06} shows that $P^{Z_i}(r)$ lies in $P^{Z_i}(T)$,
justifying our contention.

We observe that $Z_i$ inherits from $Z$ the property that it has no closed convex proper subset of full
boundary. In conclusion, since $P^{Z_i}(T)$ is a convex set, it is dense in $Z$. However, according to Lemma~5.1
in~\cite{FoertschLytchak06}, it splits as $P^{Z_i}(T) = P^{Z_i}(F_z) \times P^{Z_i}(S')$. Upon possibly
replacing $P^{Z_i}(S')$ by its completion (whilst $P^{Z_i}(F_z)$ is already closed in $Z_i$ since $P^{Z_i}$ is
isometric on $(Z_i)_z$), we obtain a splitting of the closure of $P^{Z_i}(T)$, and hence of $Z_i$. This
completes the proof of the lemma since $Z_i$ is irreducible.
\end{proof}

Now the main argument runs by induction over $m\geq 2$. Lemma~\ref{lem:FL_Dichotomy} identifies by induction
$Z_i$ with some $Z'_j$. Indeed, Lemma~\ref{lem:EuclideanDecomposition} excludes that all pairwise intersections
reduce to a point since $Z$ has no Euclidean factor.
\end{proof}

%================================================================================
\subsection{Proof of Theorem~\ref{thm:Decomposition} and Addendum~\ref{addendum}}\label{sec:decompositions}
The following consequence of the solution to Hilbert's fifth problem belongs to the mathematical lore.

\begin{thm}\label{thm:folklore}
Let $G$ be a locally compact group with trivial amenable radical. Then $G$ possesses a canonical finite index
open normal subgroup $G^\dag$ such that $G^\dag = L \times D$, where $L$ is a connected semi-simple Lie group
with trivial centre and no compact factors, and $D$ is totally disconnected.
\end{thm}

\begin{proof}
This follows from the Gleason--Montgomery--Zippin solution to Hilbert's fifth problem and the fact that
connected semi-simple Lie groups have finite outer automorphism groups. More details may be found for example
in~\cite[\S\,11.3]{Monod_LN}.
\end{proof}

Combining Theorem~\ref{thm:folklore} with Theorem~\ref{thm:geometric_simplicity}, we find the statement
given as Corollary~\ref{cor:Isom(irreducible)} in the Introduction.

\begin{thm}\label{thm:IrreducibleGroup}
Let $X\neq \RR$ be an irreducible proper \cat space with finite-dimensional Tits boundary and $G<\Isom(X)$ any
closed subgroup whose action is minimal and does not have a global fixed point in $\bd X$.

Then $G$ is either totally disconnected or an almost connected simple Lie group with trivial centre.
\end{thm}

\begin{proof}
By Theorem~\ref{thm:geometric_simplicity}, $G$ has trivial amenable radical. Let $G^\dag$ be as in
Theorem~\ref{thm:folklore}. Applying Theorem~\ref{thm:geometric_simplicity} to this normal subgroup of $G$,
deduce that we have either $G^\dag = L$ with $L$ simple or $G^\dag = D$.
\end{proof}

We can now complete the proof of Theorem~\ref{thm:Decomposition} and Addendum~\ref{addendum} and we adopt their notation.
Since $\Isom(X)$ has no global fixed point at infinity, there is a canonical minimal non-empty closed convex
$\Isom(X)$-invariant subset $X'\se X$ (Remarks~39 in~\cite{Monod_superrigid}). We apply Corollary~\ref{cor:preliminary}
to $Z=X'$ and Remark~\ref{rem:preliminary} to $G=\Isom(Z)$, setting
$$G^* = \Isom(\RR^n) \times \Isom(Z_1) \times\cdots\times \Isom(Z_m).$$
All the claimed properties of the resulting factor groups are established in
Theorem~\ref{thm:geometric_simplicity}, Theorem~\ref{thm:bootstrap} and Theorem~\dref{thm:density-intro}
(the proof of which is completely independent from the present considerations).
Finally, the claim that  any product decomposition of $G^*$ is a regrouping of the factors
in~\eqref{eq:GroupDecomposition} is established as follows. Notice that the $G^*$-action on $Z$ is still minimal
and without fixed point at infinity (this is almost by definition but alternatively also follows from
Theorem~\ref{thm:geometric_simplicity}). Therefore, given any product decomposition of $G^*$, we can apply the
splitting theorem (Corollary~10 in~\cite{Monod_superrigid}) and obtain a corresponding splitting of $Z$. Now the
uniqueness of the decomposition of the space $Z$ (away from the Euclidean factor) implies that the given
decomposition of $G^*$ is a regrouping of the factors occurring in Remark~\ref{rem:preliminary}.\hfill\qedsymbol

%================================================================================
\subsection{\cat spaces without Euclidean factor}
For the sake of future references, we record the following consequence of the results obtained thus far:

\begin{cor}\label{cor:NoEuclideanFactor}
Let $X$ be a proper \cat space with finite-dimensional boundary and no Euclidean factor, such that $G =
\Isom(X)$ acts minimally without fixed point at infinity. Then $G$ has trivial amenable radical and any subgroup
of $G$ acting minimally on $X$ has trivial centraliser. Furthermore, given a non-trivial normal subgroup $N\lhd G$,
any $N$-minimal $N$-invariant closed subspace of $X$ is a regrouping of factors in the decomposition of
Addendum~\ref{addendum}. In particular, if each irreducible factor of $G$ is non-discrete, then $G$ has no
non-trivial finitely generated closed normal subgroup.
\end{cor}

\begin{proof}
The triviality of the amenable radical comes from the corresponding statement in irreducible factors of $X$, see
Theorem~\ref{thm:geometric_simplicity}. By the second paragraph of the proof of
Theorem~\ref{thm:geometric_simplicity}, any subgroup of $G$ acting minimally has trivial centraliser. The fact
that minimal invariant subspaces for normal subgroups are fibres in the product decomposition
(\ref{eq:SpaceDecomposition}) follows since any product decomposition of $X$ is a regrouping of factors in
(\ref{eq:SpaceDecomposition}) and since any normal subgroup of $G$ yields such a product decomposition by
Theorem~\ref{thm:dichotomy}(B.i) and (B.ii). Assume finally that each irreducible factor in
(\ref{eq:GroupDecomposition}) is non-discrete and let $N< G$ be a  finitely generated closed normal subgroup.
Then $N$ is discrete by Baire's category theorem, and $N$ acts minimally on a fibre, say $Y$, of the space
decomposition (\ref{eq:SpaceDecomposition}). Therefore, the projection of $N$ to $\Isom(Y)$ has trivial
centraliser, unless $N$ is trivial. Since $N$ is discrete, normal and finitely generated, its centraliser is
open. Since $\Isom(Y)$ is non-discrete by assumption, we deduce that $N$ is trivial, as desired.
\end{proof}

%%%%%%%%%%%%%%%%%%%%%%%%%%%%%%%%%%%%%%%%%%%%%%%%%%%%%%%%%%%%%%%%%%%%%%%%%%%%%%
\section{Totally disconnected group actions}
%%%%%%%%%%%%%%%%%%%%%%%%%%%%%%%%%%%%%%%%%%%%%%%%%%%%%%%%%%%%%%%%%%%%%%%%%%%%%%

%================================================================================
\subsection{Smoothness}
When considering actions of totally disconnected groups, a desirable property is
\textbf{smoothness}\index{smooth}, namely that points have open stabilisers. This condition is important in
representation theory, but also in our geometric context, see point~\eqref{pt:BasicTotDisc:smoothss} of
Corollary~\ref{cor:BasicTotDisc} below and~\cite{CapraceTD}.

In general, this condition does not hold, even for actions that are cocompact, minimal and without fixed point
at infinity. An example will be constructed in Section~\dref{sec:example}. However, we establish it under a
rather common additional hypothesis. Recall that a metric space $X$ is called \textbf{geodesically
complete}\index{cat(0) space@\cat space!geodesically complete}\index{geodesically complete|see{\cat space}} (or
said to have \textbf{extensible geodesics}\index{extensible geodesic}) if every geodesic segment of positive
length may be extended to a locally isometric embedding of the whole real line. The following contains
Theorem~\ref{thm:smooth:intro} from the Introduction.

\begin{thm}\label{thm:smooth}
Let $G$ be a totally disconnected locally compact group with a minimal, continuous and proper action by
isometries on a proper \cat space $X$.

If $X$ is geodesically complete, then the action is smooth. In fact, the pointwise stabiliser of every bounded
set is open.
\end{thm}

\begin{remark}\label{rem:finite:isotropy}
In particular, the stabiliser of a point acts as a \emph{finite} group of isometries on any given ball around
this point in the setting of Theorem~\ref{thm:smooth}.
\end{remark}

\begin{cor}\label{cor:BasicTotDisc}
Let $X$ be a proper \cat space and $G$ be a totally disconnected locally compact group acting continuously
properly on $X$ by isometries. Then:
\begin{enumerate}
\item If the $G$-action is cocompact, then every element of zero translation length is
elliptic.\label{pt:BasicTotDisc:elliptic}

\item If the $G$-action is cocompact and every point $x \in X$ has an open stabiliser, then the $G$-action is
semi-simple.\label{pt:BasicTotDisc:smoothss}

\item If the $G$-action is cocompact and $X$ is geodesically complete, then the $G$-action is
semi-simple.\label{pt:BasicTotDisc:completess}
\end{enumerate}
\end{cor}

\begin{proof}[Proof of Corollary~\ref{cor:BasicTotDisc}]
Points~\eqref{pt:BasicTotDisc:elliptic} and~\eqref{pt:BasicTotDisc:smoothss} follow readily from
Theorem~\ref{thm:smooth}, see~\cite[Corollary~3.3]{CapraceTD}.

\medskip

\eqref{pt:BasicTotDisc:completess}~In view of Lemma~\ref{lem:cocompact:minimal}, this follows from
Theorem~\ref{thm:smooth} and~\eqref{pt:BasicTotDisc:smoothss}.
\end{proof}

The following is a key fact for Theorem~\ref{thm:smooth}:

\begin{lem}\label{lem:GeodComplete:dense}
Let $X$ be a geodesically complete proper \cat space. Let $(C_n)_{n \geq 0}$ be an increasing sequence of closed
convex subsets whose union $C = \bigcup_n C_n$ is dense in $X$.

Then every bounded subset of $X$ is contained in some $C_n$; in particular, $C = X$.
\end{lem}

\begin{proof}
Suppose for a contradiction that for some $r>0$ and $x\in X$ the $r$-ball around $x$ contains an element $x_n$
not in $C_n$ for each $n$. We shall construct inductively a sequence $\{c_k\}_{k\geq 1}$ of pairwise
$r$-disjoint elements in $C$ with $d(x, c_k)\leq 2r+2$, contradicting the properness of $X$.

If $c_1, \ldots, c_{k-1}$ have been constructed, choose $n$ large enough to that $C_n$ contains them all and
$d(x, C_n)\leq 1$. Consider the (non-trivial) geodesic segment from $x_n$ to its nearest point projection
$\overline{x_n}$ on $C_n$; by geodesic completeness, it is contained in a geodesic line and we choose $y$ at
distance $r+1$ from $C_n$ on this line. Notice that $x_n\in[\overline{x_n}, y]$ and hence $d(y, x)\leq 2r+1$.
Moreover, $d(y, c_i)\geq r+1$ for all $i<k$. Since $C$ is dense, we can choose $c_k$ close enough to $y$ to
ensure $d(c_k, x)\leq 2r+2$ and $d(c_k, c_i) \geq r$ for all $i<k$, completing the induction step.
\end{proof}

\begin{proof}[End of proof of Theorem~\ref{thm:smooth}]
The subset $C \subseteq X$ consisting of those points $x \in X$ such that the stabiliser $G_x$ is open is
clearly convex and $G$-stable. By~\cite[III \S\,4 No~6]{BourbakiTGI}, the group  $G$ contains a compact open
subgroup and hence $C$ is non-empty. Thus $C$ is dense by minimality of the action. Since $\Isom(X)$ is second
countable, we can choose a descending chain $Q_n<G$ of compact open subgroups whose intersection acts trivially
on $X$. Therefore, $C$ may be written as the union of an ascending family of closed convex subsets $C_n
\subseteq X$, where $C_n$ is the fixed point set of $Q_n$. Now the statement of the theorem follows from
Lemma~\ref{lem:GeodComplete:dense}.
\end{proof}

%================================================================================
\subsection{Locally finite equivariant partitions and cellular decompositions}
Let $X$ be a locally finite cell complex and $G$ be its group of cellular automorphisms, endowed with the
topology of pointwise convergence on bounded subsets. Then $G$ is a totally disconnected locally compact group
and every bounded subset of $X$ has an open pointwise stabiliser in $G$. One of the interest of
Theorem~\ref{thm:smooth} is that it allows for a partial converse to the latter statement:

\begin{prop}\label{prop:EquivariantPartition}
Let $X$ be a proper \cat space and $G$ be a totally disconnected locally compact group acting continuously
properly on $X$ by isometries. Assume that the pointwise stabiliser of  every bounded subset of $X$ is open in
$G$. Then we have the following:
\begin{enumerate}
\item $X$ admits a canonical locally finite $G$-equivariant partition.

\item Denoting by $\sigma(x)$ the piece supporting the point $x \in X$ in that partition, we have
$\Stab_G(\sigma(x)) = \norma_G(G_x)$ and $\norma_G(G_x)/G_x$ acts freely on $\sigma(x)$.

\item If $G \backslash X$ is compact, then so is $\Stab_G(\sigma(x)) \backslash \sigma(x)$ for all $x \in X$.
\end{enumerate}
\end{prop}

\begin{proof}
Consider the equivalence relation on $X$ defined by
$$ x \sim y \hspace{.5cm} \Leftrightarrow \hspace{.5cm}  G_x = G_y.$$
This yields a canonical $G$-invariant partition of $X$. We need to show that it is locally finite.
Assume for a contradiction that there exists a converging sequence $\{x_n\}_{n \geq 0}$ such that the subgroups
$G_{x_n}$ are pairwise distinct. Let $x = \lim_n x_n$.

We claim that $G_{x_n} < G_x$ for all sufficiently large $n$. Indeed, upon extracting there would otherwise
exist a sequence $g_n \in G_{x_n}$ such that $g_n.x \ne x$ for all $n$. Upon a further extraction, we may assume
that $g_n$ converges to some $g \in G$. By construction $g$ fixes $x$. Since $G_x$ is open by hypothesis, this
implies that $g_n$ fixes $x$ for sufficiently large $n$, a contradiction. This proves the claim.

By hypothesis the pointwise stabiliser of any ball centred at $x$ is open. Thus $G_x$ possesses a compact open
subgroup $U$ which fixes every $x_n$. This implies that we have the inclusion $U < G_{x_n} < G_x$ for all $n$.
Since the index of $U$ in $G_x$ is finite, there are only finitely many subgroups of $G_x$ containing $U$. This
final contradiction finishes the proof of (i).

\smallskip \noindent (ii) Straightforward in view of the definitions.

\smallskip \noindent (iii) Suppose for a contradiction that $H \backslash \sigma(x)$ is not
compact, where $H =\Stab_G(\sigma(x))$. Let then $y_n \in \sigma(x)$ be a sequence such that $d(y_n,H.x)> n $.
Let now $g_n \in G$ be such that $\{g_n.y_n\}$ is bounded, say of diameter $C$. By (i), the set $\{ g_n G_{y_n}
g_n\inv\}$ is thus finite. Upon extracting, we shall assume that it is constant. Now, for all $n< k$, the
element $g_n\inv g_k$ normalises $G_{y_k} = G_x$ and maps $y_k$ to a point at distance $\leq C$ from $y_n$. In
view of (ii), this is absurd.
\end{proof}

\begin{remark}
The partition of $X$ constructed above is non-trivial whenever $G$ does not act freely. This is for example the
case whenever $G$ is non-discrete and acts faithfully.
\end{remark}

The pieces in the above partition are generally neither bounded (even if $G \backslash X$ is compact), nor
convex, nor even connected. However, if one assumes that the space admits a sufficiently large amount of
symmetry, then one obtains a partition which deserves to be viewed as an equivariant cellular decomposition.

\begin{cor}\label{cor:CellularDecomposition}
Let $X$ be a proper \cat space and $G$ be a totally disconnected locally compact group acting continuously
properly on $X$ by isometries. Assume that the pointwise stabiliser of  every bounded subset of $X$ is open in
$G$, and that no open subgroup of $G$ fixes a point at infinity. Then $X$ admits admits a canonical locally
finite $G$-equivariant decomposition into compact convex pieces.
\end{cor}

\begin{proof}
For each $x \in X$, let $\tau(x)$ be the fixed-point-set of $G_x$. Then $\tau(x)$ is clearly convex; it is
compact by hypothesis. Furthermore the map $x \mapsto \tau(x)$ is $G$-equivariant. The fact that the collection
$\{\tau(x) \; | \; x \in X\}$ is locally finite follows from Proposition~\ref{prop:EquivariantPartition}.
\end{proof}

%================================================================================
\subsection{Alexandrov angle rigidity}
A further consequence of Theorem~\ref{thm:smooth} is a phenomenon of \textbf{angle
rigidity}\index{angle!rigidity}. Given an elliptic isometry $g$ of complete a \cat space $X$ and a point $x\in
X$, we denote by $c_{g,x}$ the projection of $x$ on the closed convex set of $g$-fixed points.

\begin{prop}\label{prop:angle_rigidity}
Let $G$ be a totally disconnected locally compact group with a continuous and proper cocompact action by
isometries on a geodesically complete proper \cat space $X$. Then there is $\vareps>0$ such that for any
elliptic $g\in G$ and any $x\in X$ with $gx\neq x$ we have $\aangle{c_{g,x}}{gx}{x} \geq \vareps$.
\end{prop}

(We will later also prove an angle rigidity for the Tits angle, see Proposition~\ref{prop:discrete:orbits}.)

\begin{proof}
First we observe that this bound on the Alexandrov angle is really a local property at $c_{g,x}$ of the germ of
the geodesic $[c_{g,x}, x]$ since for any $y\in[c_{g,x}, x]$ we have $c_{g,y}=c_{g,x}$.

Next, we claim that for any $n\in \NN$, any isometry of order~$\leq n$ of any complete \cat space $B$ satisfies
$\aangle{c_{g,x}}{gx}{x} \geq 1/n$ for all $x\in B$ that are not $g$-fixed. Indeed, it follows from the
definition of Alexandrov angles (see~\cite[II.3.1]{Bridson-Haefliger}) that for any $y\in [c_{g,x}, x]$ we have
$$d(gy, y) \leq d(c_{g,x}, y) \aangle{c_{g,x}}{gx}{x}.$$
Therefore, if $\aangle{c_{g,x}}{gx}{x} < 1/n$, the entire $g$-orbit of $y$ would be contained in a ball around
$y$ not containing $c_{g,x}=c_{g,y}$. This is absurd since the circumcentre of this orbit is a $g$-fixed point.

\smallskip

In order to prove the proposition, we now suppose for a contradiction that there are sequences $\{g_n\}$ of
elliptic elements in $G$ and $\{x_n\}$ in $X$ with $g_n x_n\neq x_n$ and $\aangle{c_n}{g_n x_n}{x_n} \to 0$,
where $c_n=c_{g_n,x_n}$. Since the $G$-action is cocompact, there is (upon extracting) a sequence $\{h_n\}$ in
$G$ such that $h_n c_n$ converges to some $c\in X$. Upon conjugating $g_n$ by $h_n$, replacing $x_n$ by $h_n
x_n$ and $c_n$ by $h_n c_n$, we can assume $c_n\to c$ without loosing any of the conditions on $g_x$, $x_n$ and
$c_n$, including the relation $c_n =c_{g_n,x_n}$.

Since $d(g_n c, c)\leq 2 d(c_n, c)$, we can further extract and assume that $\{g_n\}$ converges to some limit
$g\in G$; notice also that $g$ fixes $c$. By Lemma~\ref{lem:cocompact:minimal}, the action is minimal and hence
Theorem~\ref{thm:smooth} applies. Therefore, we can assume that all $g_n$ coincide with $g$ on some ball $B$
around $c$ and in particular preserve $B$. Using Remark~\ref{rem:finite:isotropy}, this provides a
contradiction.
\end{proof}

A first consequence is an analogue of a result that E.~Swenson proved for discrete groups (Theorem~11
in~\cite{Swenson99}).

\begin{cor}\label{cor:exist:hyp}
Let $G$ be a totally disconnected locally compact group with a continuous and proper cocompact action by
isometries on a geodesically complete proper \cat space $X$ not reduced to a point.

Then $G$ contains hyperbolic elements (thus in particular elements of infinite order).
\end{cor}

Beyond the totally disconnected case, we can appeal to Theorem~\ref{thm:Decomposition} and
Addendum~\ref{addendum} and state the following.

\begin{cor}\label{cor:exist:nontorsion}
Let $G$ be any locally compact group with a continuous and proper cocompact action by isometries on a
geodesically complete proper \cat space $X$ not reduced to a point.

Then $G$ contains elements of infinite order; if moreover $(\bd X)^G=\varnothing$, then $G$ contains hyperbolic
elements.
\end{cor}

\begin{proof}[Proof of Corollary~\ref{cor:exist:hyp}]
Proposition~\ref{prop:angle_rigidity} allows us use the argument form~\cite{Swenson99}: We can choose a geodesic
ray $r:\RR_+\to X$, an increasing sequence $\{t_i\}$ going to infinity in $\RR_+$ and $\{g_i\}$ in $G$ such that
the function $t\mapsto g_i r(t + t_i)$ converges uniformly on bounded intervals (to a geodesic line). For $i<j$
large enough, the angle $\aangle{h(r(t_i))}{r(t_i)}{h^2(r(t_i))}$ defined with $h=g_i\inv g_j$ is arbitrarily
close to~$\pi$. In order to prove that $h$ is hyperbolic, it suffices to show that this angle will eventually
equal~$\pi$. Suppose this does not happen; by
Corollary~\ref{cor:BasicTotDisc}\eqref{pt:BasicTotDisc:completess}, we can assume that $h$ is elliptic. We set
$x=r(t_i)$ and $c=c_{h, x}$. Considering the congruent triangles $(c, x, hx)$ and $(c, hx, h^2x)$, we find that
$\aangle{c}{x}{hx}$ is arbitrarily small. This is in contradiction with Proposition~\ref{prop:angle_rigidity}.
\end{proof}

\begin{proof}[Proof of Corollary~\ref{cor:exist:nontorsion}]
If the connected component $G^\circ$ is non-trivial, then it contains elements of infinite order; if it is
trivial, we can apply Corollary~\ref{cor:exist:hyp}.

Assume now $(\bd X)^G=\varnothing$. Then Theorem~\ref{thm:Decomposition} and Addendum~\ref{addendum} apply.
Therefore, we obtain hyperbolic elements either from Corollary~\ref{cor:exist:hyp} or from the fact that any
non-compact semi-simple group contains elements that are algebraically hyperbolic, combined with the fact that
the latter act as hyperbolic isometries. That fact is established in
Theorem~\ref{thm:algebraic}\eqref{pt:algebraic} below, the proof of which is independent of
Corollary~\ref{cor:exist:nontorsion}.
\end{proof}

%================================================================================
\subsection{Algebraic structure}
Given a topological group $G$, we define its \textbf{socle}\index{socle} $\soc(G)$ as the subgroup generated by
all minimal non-trivial closed normal subgroups of $G$. Notice that $G$ might have no minimal non-trivial closed
normal subgroup, in which case its socle is trivial.

We also recall that the \textbf{quasi-centre}\index{quasi-centre} of
a locally compact group $G$ is the subset $\QZ(G)$ consisting of all those elements possessing an open
centraliser. Clearly $\QZ(G)$ is a (topologically) characteristic subgroup of $G$. Since any element with a
discrete conjugacy class possesses an open centraliser, it follows that the quasi-centre contains all discrete
normal subgroups of $G$.

\begin{prop}\label{prop:socle}
Let $X$ be a proper \cat space without Euclidean factor and $G < \Isom(X)$ be a closed subgroup acting minimally
cocompactly without fixed point at infinity. If $G$ has trivial quasi-centre, then $\soc(G^*)$ is direct
product of $r$ non-trivial characteristically simple groups, where $r$ is the number of irreducible factors of $X$ and $G^*$
is the canonical finite index open normal subgroup acting trivially on the set of factors of $X$.
\end{prop}

The proof will use the following general fact inspired by a statement for tree automorphisms, Lemma~1.4.1 in~\cite{Burger-Mozes1}.

\begin{prop}\label{prop:FilterNormalSubgroups}
Let $G$ be a compactly generated totally disconnected locally compact group without non-trivial compact normal
subgroups. Then any filtering family of non-discrete closed normal subgroups has non-trivial (thus non-compact)
intersection.
\end{prop}

A variant of this proposition is proved in~\cite{Caprace-Monod_monolith}; since the proof is short, we give it for
the sake of completeness.

\begin{proof}[Proof of Proposition~\ref{prop:FilterNormalSubgroups}]
Let $\mathfrak{g}$ be a \textbf{Schreier graph}\index{Schreier graph} for $G$. We recall that it consists in
choosing any open compact subgroup $U<G$ (which exists by~\cite[III \S\,4 No~6]{BourbakiTGI}), defining the vertex
set of $\mathfrak g$ as $G/U$ and drawing edges according to a compact generating set which is a union of double
cosets modulo $U$; see~\cite[\S\,11.3]{Monod_LN}. Since $G$ has no non-trivial compact normal subgroup, the
continuous $G$-action on $\mathfrak{g}$ is faithful. Let $v_0$ be a vertex of $\mathfrak g$ and denote by
$v_0^\perp$ the set of neighbouring vertices. Since $G$ is vertex-transitive on $\mathfrak{g}$, it follows that
for any normal subgroup $N \lhd G$, the $N_{v_0}$-action on $v_0^\perp$ defines a finite permutation group $F_N
< \Sym(v_0^\perp)$ which, as an abstract permutation group, is independent of the choice of $v_0$. Therefore, if
$N$ is non-discrete, this permutation group $F_N$ has to be non-trivial since $U$ is open and $\mathfrak g$
connected. Now a filtering family $\mathscr{F}$ of non-discrete normal subgroups yields a filtering family of
non-trivial finite subgroups of $\Sym(v_0^\perp)$. Thus the intersection of these finite groups is non-trivial.
Let $g$ be a non-trivial element in this intersection. For any $N \in \mathscr{F}$, let $N_g$ be the inverse
image of $\{g\}$ in $N_{v_0}$. Thus $N_g$ is a non-empty compact subset of $N$ for each $N \in \mathscr{F}$.
Since the family $\mathscr{F}$ is filtering, so are $\{N_{v_0} \; | \; N \in \mathscr{F} \}$ and $\{N_g \; | \;
N \in \mathscr{F} \}$. The result follows, since a filtering family of non-empty closed subsets of the compact
set $G_{v_0}$ has a non-empty intersection.
\end{proof}

Evidently open normal subgroups form a filtering family; we can thus deduce:

\begin{cor}\label{cor:ResiduallyDiscrete}
Let $G$ be a compactly generated locally compact group without any non-trivial compact normal subgroup. If $G$
is residually discrete, then it is discrete.\qed
\end{cor}

\begin{proof}[Proof of Proposition~\ref{prop:socle}]
We first observe that $G^*$ has no non-trivial discrete normal subgroup. Indeed, such a subgroup has finitely many
$G$-conjugates, which implies that each of its elements has discrete $G$-conjugacy class and hence belongs to
$\QZ(G)$, which was assumed trivial.

Let now $\{N_i\}$ be a chain of non-trivial closed normal subgroups of $G^*$. If  $N_i$ is totally disconnected for
some $i$, then the intersection $\bigcap_i N_i$ is non-trivial by Proposition~\ref{prop:FilterNormalSubgroups}.
Otherwise $N_i^\circ$ is non-trivial and normal in $(G^*)^\circ$ for each $i$, and the intersection $\bigcap_i
N_i$ is non-trivial by Theorem~\ref{thm:Decomposition} (since the latter describes in particular the possible normal connected
subgroups of $G^*$). In all cases, Zorn's lemma implies that the ordered set
of non-trivial closed normal subgroups of $G^*$ possesses minimal elements.

Given two minimal closed normal subgroups $M, M'$, the intersection $M \cap M'$ is thus trivial and, hence, so
is $[M, M']$. Thus minimal closed normal subgroups of $G^*$ centralise one another. We deduce from
Corollary~\ref{cor:NoEuclideanFactor} that the number of minimal closed normal subgroups is at most $r$.

Consider now an irreducible totally disconnected factor $H$ of $G^*$. We claim that the collection of non-trivial
closed normal subgroups of $H$ forms a filtering family. Indeed, given two such normal subgroup $N_1, N_2$, then
$N_1 \cap N_2$ is again a closed normal subgroup of $H$. It is is trivial, then the commutator $[N_1, N_2]$ is
trivial and, hence, the centraliser of $N_1$ in $H$ is non-trivial, contradicting
Theorem~\ref{thm:geometric_simplicity}. This confirms the claim. Thus the intersection of all non-trivial closed
normal subgroups of $H$ is non-trivial by Proposition~\ref{prop:FilterNormalSubgroups}. Clearly this intersection
is the socle of $H$; it is clear we have just established that it is contained in every non-trivial closed normal
subgroup of $H$.  In particular $\soc(H)$ is characteristically simple. The desired result follows, since
$\soc(H)$ is clearly a minimal closed normal subgroup of $G^*$.
\end{proof}

\begin{thm}\label{thm:TotDiscIrreducible}
Let $X$ be a proper irreducible geodesically complete \cat space. Let $G < \Isom(X)$ be a closed totally
disconnected subgroup acting cocompactly, in such a way that no open subgroup fixes a point at infinity. Then we
have the following:
\begin{enumerate}
\item Every compact subgroup of $G$ is contained in a maximal one; the maximal compact subgroups fall into
finitely many conjugacy classes.

\item $\QZ(G) = 1$.

\item $\soc(G)$ is a non-discrete characteristically simple group.
\end{enumerate}
\end{thm}

\begin{proof}
(i) By Lemma~\ref{lem:cocompact:minimal}, the action is minimal and hence Theorem~\ref{thm:smooth} applies.
In particular, we can apply Corollary~\ref{cor:CellularDecomposition} and consider the resulting
equivariant decomposition. Let $Q < G$ be
a compact subgroup and $x$ be a $Q$-fixed point. If $G_x$ is not contained in a maximal compact subgroup of $G$,
then there is an infinite sequence $(x_n)_{n \geq 0}$ such that $x_0 = x$ and $G_{x_n} \se G_{x_{n+1}}$. By
Corollary~\ref{cor:CellularDecomposition}, the sequence $x_n$ leaves every bounded subset. Since the fixed
points $ X^{G_{x_n}}$ form a nested sequence, it follows that $X^{G_x}$ is unbounded. In particular its visual
boundary $\bd (X^{G_x})$ is non-empty and the open subgroup $G_x$ has a fixed point at infinity. This contradicts
the hypotheses, and the claim is proved. Notice that a similar argument shows that for each $x \in X$, there are
finitely many maximal compact subgroups $Q_i < G$ containing $G_x$.

The fact that $G$ possesses finitely many conjugacy classes of maximal compact subgroups now follows from the
compactness of $G\backslash X$.

\smallskip \noindent (ii) We claim that $\QZ(G)$ is \textbf{topologically locally finite}, which means that
every finite subset of it is contained in a compact subgroup. The desired result follows since
it is then amenable but $G$ has trivial amenable radical by Theorem~\ref{thm:geometric_simplicity}. Let $S \se \QZ(G)$ be a finite subset.
Then $G$ possesses a compact open subgroup $U$ centralising $S$. By hypothesis the fixed point set of $U$ is
compact. Since $\la S \ra$ stabilises $X^U$, it follows that $\overline{\la S \ra}$ is compact, whence the
claim.

\smallskip \noindent (iii) Follows from (ii) and Proposition~\ref{prop:socle}.
\end{proof}

%%%%%%%%%%%%%%%%%%%%%%%%%%%%%%%%%%%%%%%%%%%%%%%%%%%%%%%%%%%%%%%%%%%%%%%%%%%%%%
\section{Cocompact \cat spaces}\label{sect:cocompact}
%%%%%%%%%%%%%%%%%%%%%%%%%%%%%%%%%%%%%%%%%%%%%%%%%%%%%%%%%%%%%%%%%%%%%%%%%%%%%%
%================================================================================
\subsection{Fixed points at infinity}
We begin with a simple observation. We recall that two points at infinity are \textbf{opposite}\index{opposite}
if they are the two endpoints of a geodesic line. We denote by $\xi\opp$ the set of points opposite to $\xi$.
Recall from~\cite[Theorem~4.11(i)]{BallmannLN} that, if $X$ is proper, then two points  $\xi, \eta \in \bd X$ at
Tits distance~$>\pi$ are necessarily opposite. (Recall that Tits distance is by definition the length metric
associated to the Tits angle.)\index{Tits!distance}
However, it is not true in general that two points at Tits distance~$\pi$ are opposite.

\begin{prop}\label{prop:fixed:cocompact}
Let $X$ be a proper \cat space and $H<\Isom(X)$ a closed subgroup acting cocompactly. If $H$ fixes a point $\xi$
at infinity, then $\xi\opp\neq\varnothing$ and $H$ acts transitively on $\xi\opp$.
\end{prop}

\begin{proof}
First we claim that there is a geodesic line $\sigma:\RR\to X$ with $\sigma(\infty)=\xi$. Indeed, let
$r:\RR_+\to X'$ be a ray pointing to $\xi$ and $\{g_n\}$ a sequence in $H$ such that $g_n r(n)$ remains bounded.
The Arzel\`a--Ascoli theorem implies that $g_n r(\RR_+)$ subconverges to a geodesic line in $X$. Since $\xi$ is
fixed by all $g_n$, this line has an endpoint at $\xi$.

Let now $\sigma':\RR\to X$ be any other geodesic with $\sigma'(\infty)=\xi$ and choose a sequence
$\{h_n\}_{n\in\NN}$ in $H$ such that $d(h_n \sigma(-n), \sigma'(-n))$ remains bounded. By convexity and since all
$h_n$ fix $\xi$, $d(h_n\sigma(t), \sigma'(t))$ is bounded for all $t$ and thus subconverges (uniformly for $t$
in bounded intervals). On the one hand, it implies that $\{h_n\}$ has an accumulation point $h$. On the other
hand, it follows that $h\sigma(-\infty) = \sigma'(-\infty)$.
\end{proof}

Recall that any complete \cat space $X$ admits a canonical splitting $X=X'\times V$ preserved by all isometries,
where $V$ is a (maximal) Hilbert space called the Euclidean factor of $X$,
see~\cite[II.6.15(6)]{Bridson-Haefliger}. Furthermore, there is a canonical embedding $X'\se X''\times V'$,
where $V'$ is a Hilbert space generated by all directions in $X'$ pointing to ``flat points'' at infinity,
namely points for which the Busemann functions are affine on $X'$; moreover, every isometry of $X'$ extends
uniquely to an isometry of $X''\times V'$ which preserves that splitting. This is a result of
Adams--Ballmann~\cite[Theorem~1.6]{AB98}, who call $V'$ the \textbf{pseudo-Euclidean
factor}\index{pseudo-Euclidean factor|see{Euclidean}}
(one could also propose ``Euclidean pseudo-factor'').\index{Euclidean!pseudo-factor}

\begin{cor}\label{cor:pseudoEuclidean}
Let $X$ be a proper \cat space with a cocompact group of isometries. Then the pseudo-Euclidean factor of $X$ is
trivial.
\end{cor}

\begin{proof}
In view of the above discussion, $X'$ is also a proper \cat space with a cocompact group of isometries. The set
of flat points in $\bd X'$ admits a canonical (intrinsic) circumcentre $\xi$ by Lemma~1.7 in~\cite{AB98}. In
particular, $\xi$ is fixed by all isometries and therefore, by Proposition~\ref{prop:fixed:cocompact}, it has an
opposite point, which is impossible for a flat point unless it lies already in the Euclidean factor
(see~\cite{AB98}).
\end{proof}

\begin{prop}\label{prop:StabIsMinimal}
Let $G$ be a group acting cocompactly by isometries on a proper \cat space $X$ without Euclidean factor and
assume that the stabiliser of every point at infinity acts minimally on $X$. Then $G$ has no fixed point at
infinity.
\end{prop}

\begin{proof}
If $G$ has a global fixed point $\xi$, then the stabiliser $G_\eta$ of an opposite point $\eta\in\xi\opp$ (which
exists by Proposition~\ref{prop:fixed:cocompact}) preserves the union $Y\se X$ of all geodesic lines connecting
$\xi$ to $\eta$. By~\cite[II.2.14]{Bridson-Haefliger}, this space is convex and splits as $Y=\RR\times Y_0$.
Since $G_\eta$ acts minimally, we deduce $Y=X$ which provides a Euclidean factor.
\end{proof}

%================================================================================
\subsection{Actions of simple algebraic groups}
Let $X$ be a \cat space and $\GG$ be an algebraic group defined over the field $k$. An isometric action of
$\GG(k)$ on $X$ is called \textbf{algebraic}\index{algebraic action} if every (algebraically) semi-simple
element $g \in \GG(k)$ acts as a semi-simple isometry.

When $\GG$ is semi-simple, we denote by $X_{\mathrm{model}}$ the Riemannian symmetric space or Bruhat--Tits building associated
with $\GG(k)$.

\begin{thm}\label{thm:algebraic}
Let $k$ be a local field and $\GG$ be an absolutely almost simple simply connected $k$-group. Let $X$ be a
non-compact proper \cat space on which $G = \GG(k)$ acts continuously by isometries.

Assume either: (a)~the action is cocompact; or: (b)~it has full limit set, is minimal and $\bd X$ is finite-dimensional.
Then:

\begin{enumerate}
\item The $G$-action is algebraic.\label{pt:algebraic}

\item There is a $G$-equivariant bijection $\bd X \cong \bd X_{\mathrm{model}}$ which is an isometry with respect to Tits' metric
and a homeomorphism with respect to the cône topology.
This bijection extends to a $G$-equivariant rough isometry $\beta: X_{\mathrm{model}}\to X$. \label{pt:algebraic:bd}

%\item The \label{pt:QI}
%??? J'intègre l'énoncé QI au point (ii) afin de ne pas perturber la numérotation.

\item If $X$ is geodesically complete, then $X$ is isometric to $X_{\mathrm{model}}$.\label{pt:algebraic:homo}

\item For any semi-simple $k$-subgroup $\LL<\GG$, there non-empty closed convex subspace $Y\se X$ minimal for
$L=\LL(k)$; moreover, there is no $L$-fixed point in $\bd Y$.\label{pt:algebraic:sub}
\end{enumerate}
\end{thm}

In the above point~\eqref{pt:algebraic:bd}, a \textbf{rough isometry} refers to a map $\beta: X_{\mathrm{model}}\to X$ such that there is a constant $C$
with
$$d_{X_{\mathrm{model}}} (x, y) - C\ \leq\ d_X(\beta(x), \beta(y))\ \leq\ d_{X_{\mathrm{model}}} (x,y) +C$$
for all $x, y\in X_{\mathrm{model}}$ and such that $\beta(X_{\mathrm{model}})$ has finite codiameter in $X$. Such a map is also called a $(1,C)$-quasi-isometry.

\begin{remarks}\ % C'est juste pour aller à la ligne
\begin{enumerate}
\item Notice that there is no assumption on the $k$-rank of $\GG$ in this result.

\item We recall for~(b) that minimality follows from full limit set in the geodesically complete case
(Lemma~\ref{lem:cocompact:minimal}).

\item In the context of~\eqref{pt:algebraic:bd}, we recall that in general two \cat spaces with the same cocompact
isometry group need not have homeomorphic boundaries~\cite{Croke-Kleiner}.

\item \emph{A posteriori}, point~\eqref{pt:algebraic:bd} shows in particular that the action is also cocompact under the assumption~(b).
\end{enumerate}
\end{remarks}

Before proceeding to the proof, we give two examples showing that the assumptions made in Theorem~\ref{thm:algebraic}
are necessary.

\begin{example}\label{ex:Arbre-triangle}
Without the assumption of geodesic completeness, it is not true in general that, in the setting of the theorem,
the space $X$ contains a closed convex $G$-invariant subspace which is isometric to $X_\mathrm{model}$. A simple
example of this situation may obtained as follows. Consider the case where $k$ is non-Archimedean and $G$ has
$k$-rank one. Let $0<r<1/2$ and let $X$ be the space obtained by replacing the $r$-ball centred at each vertex
in the tree $X_\mathrm{model}$ by an isometric copy of a given Euclidean $n$-simplex, where $n+1$ is valence of
the vertex. In this way, one obtains a \cat space which is still endowed with an isometric $G$-action that is
cocompact and minimal, but clearly $X$ is not isometric to $X_\mathrm{model}$.
\end{example}

We do not know whether such a construction may also be performed in the Archimedean case (see
Problem~\dref{pb:KarpelevichMostow}).

\begin{example}\label{ex:parabolic_cone}
Under the assumptions~(b), minimality is needed. Indeed, we claim that for any \cat space $X_0$ there is a
canonical CAT($-1$) space $X$ (in particular $X$ is a \cat space) together with a canonical map
$i:\Isom(X_0)\hookrightarrow \Isom(X)$ with the following properties: The boundary $\bd X$ is reduced to a
single point; $X$ non-compact; $X$ is proper if and only if $X_0$ is so; the map $i$ is an isomorphism of
topological groups onto its image. This claim justifies that minimality is needed since we can apply it to the
case where $X_0$ is the symmetric space or Bruhat--Tits building associated to $\GG(k)$. (In that case the
action has indeed full limit set, a cheap feat as the isometry group is non-compact and the boundary rather
incapacious.)

To prove the claim, consider the \textbf{parabolic cône}\index{cone@cône!parabolic}\index{parabolic
cone@parabolic cône|see{cône}} $Y$ associated to $X_0$. This is the metric space with underlying set
$X_0\times\RR^*_+$ where the distance is defined as follows: given two points $(x, t)$ and $(x', t')$ of $Y$,
identify the interval $[x, x']\se X_0$ with an interval of corresponding length in $\RR$ and measure the length
from the resulting points $(x, t)$ and $(x', t')$ in the upper half-plane model for the hyperbolic plane. This
is a particular case of the synthetic version (\cite{Chen99},~\cite{Alexander-Bishop98}) of the Bishop--O'Neill
``warped products''~\cite{Bishop-ONeill} and its properties are described
in~\cite{Burago-Gromov-Perelman},~\cite[1.2(2A)]{Alexander-Bishop04} and~\cite[\S\,2]{Hummel-Lang-Schroeder}. In
particular, $Y$ is CAT($-1$).

We now let $\xi\in\bd Y$ be the point at infinity corresponding to $t\to\infty$ and define $X\se Y$ to be an associated horoball; for definiteness,
set $X=X_0 \times [1, \infty)$. We now have $\bd X=\{\xi\}$ by the CAT($-1$) property or alternatively by the explicit description of geodesic rays
(\emph{e.g.} 2(iv) in~\cite{Hummel-Lang-Schroeder}). The remaining properties follow readily.
\end{example}

\begin{proof}[Proof of Theorem~\ref{thm:algebraic}]
We start with a few preliminary observations. Finite-dimensionality of the boundary always holds since it is automatic in the cocompact case.
Since $X$ is non-compact, the action is non-trivial, because it
has full limit set. It is well known that every non-trivial continuous homomorphism of $G$ to a locally compact
second countable group is proper~\cite[Lemma~5.3]{BM96}. Thus the $G$-action on $X$ is proper.

We claim that the stabiliser of any point $\xi \in \bd X$ contains the unipotent radical of some proper
parabolic subgroup of $G$. Indeed, fix a polar decomposition $G = KTK$. Let $x_0 \in X$ be a $K$-fixed point.
Choose a sequence $\{g_n\}_{n \geq 0}$ of elements of $G$ such that $g_n.x_0$ converges to $\xi$. Write $g_n=
k_n.a_n.k'_n$ with $k_n, k'_n \in K$ and $a_n \in T$. We may furthermore assume, upon replacing $\{g_n\}$ by a
subsequence, that $\{k_n\}$ converges to some $k \in K$, that $\{a_n.x_0\}$ converges in $X \cup \bd X$ and that
$\{a_n.p\}$ converges in $X_{\mathrm{model}} \cup \bd X_{\mathrm{model}} $, where $p \in X_{\mathrm{model}}$ is
some base point. Let $\eta = \lim_{n \to \infty} a_n.x_0$ and observe $\eta = k\inv \xi$. Furthermore, the
stabiliser of $\eta$ contains the group
$$U = \{g \in G \; | \; \lim_{n \to \infty} a_n\inv g a_n = 1\}.$$
The convergence in direction of $\{a_n\}$ in $T$ implies that $U$ contains the unipotent radical $U_Q$ of the
parabolic subgroup $Q<G$ corresponding to $\lim_{n \to \infty} a_n.p \in \bd X_{\mathrm{model}}$. (In fact, the
arguments for Lemma~2.4 in~\cite{Prasad77} probably show $U=U_Q$; this follows \emph{a posteriori} from~(ii)
below.) Therefore, the stabiliser of $\xi = k.\eta$ in $G$ contains the unipotent radical of $kQk\inv$, proving
the claim.

Notice that we have seen in passing that any point at infinity lies in the limit set of some torus; in the above
notation, $\xi$ is in the limit set of $k T k\inv$.

\medskip

\eqref{pt:algebraic}~Every element of $G$ which is algebraically elliptic acts with a fixed point in $X$, since
it generates a relatively compact subgroup. We need to show that every non-trivial element of a maximal split
torus $T<G$ acts as a semi-simple isometry. Assume for a contradiction that some element $t \in T$ acts as a
parabolic isometry. Since $X$ has finite-dimensional boundary and we can apply
Corollary~\ref{cor:FixedPointParabolic}\eqref{pt:FixedPointParabolic:centra}. It follows that the Abelian group
$T$ has a canonical fixed point at infinity $\xi$ fixed by the normaliser $\norma_G(T)$. By the preceding
paragraph, we know furthermore that the stabiliser of $\xi$ in $G$ also contains the unipotent radical of some
parabolic subgroup of $G$. Recall that $G$ is generated by $\norma_G(T)$ together with any such unipotent
radical: this follows from the fact that $\norma_G(T)$ has no fixed point at infinity in $X_{\mathrm{model}}$
and that $G$ is generated by the unipotent radicals of any two distinct parabolic subgroups. Therefore $\xi$ is
fixed by the entire group $G$. Since $G$ has trivial Abelianisation, its image under the Busemann character
centred at $\xi$ vanishes, thereby showing that $G$ must stabilise every horoball centred at $\xi$. This is
absurd both in the minimal and the cocompact case.

\medskip

\eqref{pt:algebraic:bd}~Let $T < G$ be a maximal split torus. Let $F_{\mathrm{model}} \se
X_{\mathrm{model}}$ be the (maximal) flat stabilised by $T$. In view of~\eqref{pt:algebraic} and the properness
of the $T$-action, we know that $T$ also stabilises a flat $F \se X$ with $\dim F = \dim T$,
see~\cite[II.7.1]{Bridson-Haefliger}. Choose a base point $p_0 \in F_{\mathrm{model}}$ in such a way that
its stabiliser $K := G_{p_0}$ is a maximal compact subgroup of $G$. The union of all $T$-invariant flats which
are parallel to $F$ is $\norma_G(T)$-invariant. Therefore, upon replacing $F$ by a parallel flat, we may -- and
shall -- assume that $F$ contains a point $x_0$ which is stabilised by $N_K := \norma_G(T) \cap K$. Note that,
since $\norma_G(T) = \la N_K \cup T \ra$, the flat $F$ is $\norma_G(T)$-invariant. Therefore, there is a well
defined $\norma_G(T)$-equivariant map $\alpha$ of the $\norma_G(T)$-orbit of $p_0$ to $F$, defined by
$\alpha(g.p_0) = g.x_0$ for all $g \in \norma_G(T)$.

We claim that, up to a scaling factor, the map $\alpha$ is isometric and induces an $\norma_G(T)$-equivariant isometry
$\alpha: F_{\mathrm{model}}\to F$. In order to establish this, remark that
the Weyl group $W:= \norma_G(T)/\centra_G(T)$ acts on $F$, since $W = N_K/T_K$, where $T_K:= \centra_G(T) \cap
K$ acts trivially on $F$. The group $N_K$ normalises the coroot lattice $\Lambda < T$. Furthermore $N_K.\Lambda$
acts on $F_{\mathrm{model}}$ as an affine Weyl group since $N_K.\Lambda/T_K \cong W \ltimes \Lambda$. Moreover,
since any reflection in $W$ centralises an Abelian subgroup of corank 1 in $\Lambda$, it follows that
$N_K.\Lambda$ acts on $F$ as a discrete reflection group. But a given affine Weyl group has a unique (up to
scaling factor) discrete cocompact action as a reflection group on Euclidean spaces, as follows from
\cite[Ch.~VI, \S\,2, Proposition~8]{Bourbaki_Lie456}. Therefore the restriction of $\alpha$ to $\Lambda.x_0$ is a
homothety. Since $\Lambda$ is a uniform lattice in $T$, the claim follows.

At this point, it follows that $\alpha$ induces an $\norma_G(T)$-equivariant map $\bd\alpha : \bd
F_{\mathrm{model}} \to \bd F$, which is isometric with respect to Tits' distance. We recall that $\norma_G(T)$ is the
stabiliser of $\bd F_{\mathrm{model}}$ in $G$. Morover, for any $\eta\in \bd F_{\mathrm{model}}$, the stabilisers in
$G$ of $\eta$ is contained in that of $\alpha(\eta)$ because of the geometric description of parabolic subgroups alluded to in
the preliminary observation: see the argument for Lemma~2.4 in~\cite{Prasad77}. Therefore, $\bd \alpha$
extends to a well defined $G$-equivariant map $\bd
X_{\mathrm{model}} \to \bd X$, which we denote again by $\bd\alpha$. Since any two points of $\bd
X_{\mathrm{model}}$ are contained in a common maximal sphere (i.e. an apartment), and since $G$ acts
transitively on these spheres, the map $\bd\alpha$ is isometric, because so is its restriction to the sphere
$\bd F_{\mathrm{model}}$. Note that $\bd\alpha$ is surjective: indeed, this follows from the last preliminary
observation, which, combined with~\eqref{pt:algebraic}, shows in particular that $\bd X = K. \bd F$.

We now show that $\bd\alpha$ is a homeomorphism with respect to the cône topology. Since $\bd
X_{\mathrm{model}}$ is compact, it is enough to show that $\bd\alpha$ is continuous. Now any convergent sequence
in $\bd X_{\mathrm{model}}$ may be written as $\{k_n.\xi_n\}_{n \geq 0}$, where $\{k_n\}_{n \geq 0}$ (resp.
$\{\xi_n\}_{n \geq 0}$) is a convergent sequence of elements of $K$ (resp. $\bd F_{\mathrm{model}}$). On the
sphere $\bd F_{\mathrm{model}}$, the cône topology coincides with the one induced by Tits' metric. Therefore,
the equivariance of the Tits' isometry $\bd\alpha$  shows that $\{\bd\alpha(k_n.\xi_n)\}_{n \geq 0}$ is a
convergent sequence in $\bd X$, as was to be proved.

We next claim that that $G$-action on $X$ is cocompact even under the assumption~(b).
Towards a contradiction, assume otherwise. Choose a sequence $\{y_n\}$ in $X$ with $y_0$ a $K$-fixed point and such that
$d(y_n, g.y_0)\geq n$ for all $g \in G$. Upon replacing $y_n$ ($n\geq 1$) by an appropriate $G$-translate, we can and shall assume
that moreover
\begin{equation}\label{eq:algebraic:i}
d(y_n, y_0) \leq d(y_n, g.y_0) + c \hspace{1cm} \forall\,g \in G, n\geq 1,
\end{equation}
where $c$ is some constant. Upon extracting a subsequence, the sequence $\{y_n\}$ converges to some
point $\eta \in \bd X$. It was established above that $\bd X= K.\bd F$; in particular there exists $k \in K$
such that $k.\eta \in \bd F$. Now, upon replacing $y_n$ by $k.y_n$, we obtain a sequence $\{y_n\}$ which still
satisfies all above conditions but which converges to a boundary point $\eta'$ of the flat $F$.
Let $r:\RR_+\to F$ be a geodesic ray pointing towards $\eta'$.
Since $\norma_G(T)$ acts cocompactly on the flat $F$, it follows from (\ref{eq:algebraic:i}) that for some
constant $c'$, we have
\begin{equation}\label{eq:algebraic:ii}
d(y_n, y_0) \leq d(y_n, r(t)) + c' \hspace{1cm} \forall\,t\geq 0, n\geq 1.
\end{equation}
Fix now $s>d(y_0, r(0))+c'$. For $n$ sufficiently large, let $z_n$ be the point on $[r(0), y_n]$ at distance $s$ of $r(0)$.
We have
\begin{align*}
d(y_n, r(s)) &\leq d(y_n, z_n) + d(z_n, r(s))\\
&= d(y_n, r(0)) - s + d(z_n, r(s))\\
&< d(y_n, r(0)) - d(y_0, r(0)) - c'+ d(z_n, r(s))\\
&\leq d(y_n, y_0) -c' + d(z_n, r(s)).
\end{align*}
As $n$ goes to infinity, this provides a contradiction to~\eqref{eq:algebraic:ii} since $z_n$ converges to $r(s)$;
thus cocompactness is established.

It remains for~\eqref{pt:algebraic:bd} to prove that $\bd\alpha$ extends to a $G$-equivariant rough isometry
$\beta: X_{\mathrm{model}} \to X$.
The orbital map $g\mapsto g.y_0$  associated to $y_0$ yields a map $\beta: G/K\to X$; when $k$ is Archimedean, $X_{\mathrm{model}}=G/K$
whereas we extend $\beta$ linearly to each chamber of the building $X_{\mathrm{model}}$ in the non-Archimedean case.
It is a well-known consequence of cocompactness that the $G$-equivariant map $\beta:X_{\mathrm{model}} \to X$ is a
quasi-isometry (see \emph{e.g.} the proof of the Milnor--\v{S}varc lemma given in~\cite[I.8.19]{Bridson-Haefliger}).
For our stronger statement, it suffices, in view of the $KTK$ decomposition and of equivariance,
to prove that there is a constant $C'$ such that
$$d_{X_{\mathrm{model}}} (a.p_0, p_0) - C' \leq d_X(a.y_0, y_0) \leq d_{X_{\mathrm{model}}} (a.p_0, p_0) +C'$$
for all $a\in T$. This follows from the fact that $\beta$ and $\alpha$ are at bounded distance from each other on $F_{\mathrm{model}}$
(indeed, at distance $d(y_0, x_0)$) and that $\beta$ is isometric on $F_{\mathrm{model}}$.

\medskip

\eqref{pt:algebraic:homo}~In the higher rank case, assertion~\eqref{pt:algebraic:homo} follows
from~\eqref{pt:algebraic:bd} and the main result of~\cite{Leeb}. However, the full strength of \emph{loc.\ cit.}
is really not needed here, since the main difficulty there is precisely the absence of any group action, which
is part of the hypotheses in our setting. For example, when the ground field $k$ is the field of real numbers,
the arguments may be dramatically shortened as follows; they are valid without any rank assumption.

Given any $\xi \in \bd X$, the unipotent radical of the parabolic subgroup $G_\xi$ acts sharply transitively on
the boundary points opposite to $\xi$. In view of this and of the properness of the $G$-action, the arguments
of~\cite[Proposition~4.27]{Leeb} show that geodesic lines in $X$ do not branch; in other words $X$ has uniquely
extensible geodesics. From this, it follows that the group $N_K = \norma_G(T) \cap K$ considered in the proof
of~\eqref{pt:algebraic:bd} has a unique fixed point in $X$, since otherwise it would fix pointwise a geodesic
line, and hence, by~\eqref{pt:algebraic:bd}, opposite points in $\bd X_\mathrm{model}$. The fact that this is
impossible is purely a statement on the classical symmetric space $X_\mathrm{model}$; we give a proof for the
reader's convenience:

Let $F_\mathrm{model}$ be the flat corresponding to $T$ and $p_0 \in F_\mathrm{model}$ be the $K$-fixed point.
If $N_K$ fixed a point $\xi \in \bd X_\mathrm{model}$, then the ray $[p_0, \xi)$ would be pointwise fixed and,
hence, the group $N_K$ would fix a non-zero vector in the tangent space of $X_\mathrm{model}$ at $p_0$. A Cartan
decomposition $\mathfrak{g} = \mathfrak{k} \oplus \mathfrak{p}$ of the Lie algebra $\mathfrak{g}$ of $G$ yields
an isomorphism between the isotropy representation of $N_K$ on $T_{p_0}X_\mathrm{model}$ and the representation
of the Weyl group $W$ on $\mathfrak{p}$. An easy explicit computation shows that the latter representation has
no non-zero fixed vector.

\smallskip
Since $N_K$ has a unique fixed point, the latter is stabilised by the entire group $K$. Hence $K$ fixes a point
lying on a flat $F$ stabilised by $T$. From the $KTK$-decomposition, it follows that the $G$-orbit of this fixed
point is convex. Since the $G$-action on $X$ is minimal by geodesic completeness
(Lemma~\ref{lem:cocompact:minimal}), we deduce that $G$ is transitive on $X$. In particular $X$ is covered by
flats which are $G$-conjugate to $F$, and the existence of a $G$-equivariant homothety $X_{\mathrm{model}} \to
X$ follows from the existence of a $\norma_G(T)$-equivariant homothety $F_{\mathrm{model}} \to F$, which has
been established above. It remains only to choose the right scale on $X_{\mathrm{model}}$ to make it an
isometry.

\smallskip
In the non-Archimedean case, we consider only the rank one case, referring to~\cite{Leeb} for higher rank. Let
$K$ be a maximal compact subgroup of $G$ and $x_0 \in X$ be a $K$-fixed point. By~\eqref{pt:algebraic:bd}, the
group $K$ acts transitively on $\bd X$. Since $X$ is geodesically complete, it follows that the $K$-translates
of any ray emanating from $x_0$ cover $X$ entirely. On the other hand every point in $X$ has an open stabiliser
by Theorem~\ref{thm:smooth}, any point in $X$ has a finite $K$-orbit. This implies that the space of directions
at each point $p \in X$ is finite. In other words $X$ is $1$-dimensional. Since $X$ is \cat and locally compact,
it follows that $X$ is a locally finite metric tree. As we have just seen, the group $K$ acts transitively on
the geodesic segments of a given length emanating from $x_0$. One deduces that $G$ is transitive on the edges of
$X$. In particular all edges of $X$ have the same length, which we can assume to be as in $X_\mathrm{model}$,
$G$ has at most two orbits of vertices, and $X$ is either regular or bi-regular. The valence of any vertex $p$
equals
$$\min_{g \not \in \norma_G(G_p)} \big[ G_p : G_p \cap g G_p g\inv\big],$$
and coincides therefore with the valence of $X_\mathrm{model}$. It finally follows that $X$ and
$X_\mathrm{model}$ are isometric, as was to be proved.

\medskip
\eqref{pt:algebraic:sub}~Let $\PP$ be a $k$-parabolic subgroup of $\GG$ that is minimal amongst those containing
$\LL$. We may assume $\PP\neq\GG$ since otherwise $L$ has no fixed point in $\bd X$ and the conclusion holds in
view of Proposition~\ref{prop:EasyDichotomy}. It follows that $\LL$ centralises a $k$-split torus $\TT$ of
positive dimension $d$. It follows from~\eqref{pt:algebraic} that $T=\TT(k)$ acts by hyperbolic isometries, and
thus there is a $T$-invariant closed convex subset $Z\se X$ of the form $Z=Z_1\times \RR^d$ such that the
$T$-action is trivial on the $Z_1$ factor; this follows from Theorem~II.6.8 in~\cite{Bridson-Haefliger} and the
properness of the action. Moreover, $L$ preserves $Z$ and its decomposition $Z=Z_1\times \RR^d$, acting by
translations on the $\RR^d$ factor (\emph{loc.\ cit.}). Since $\LL$ is semi-simple, this translation action is
trivial and thus $L$ preserves any $Z_1$ fibre, say for instance $Z_0:=Z_1\times\{0\} \se Z$. For both the
existence of a minimal set $Y$ and the condition $(\bd Y)^L=\varnothing$, it suffices to show that $L$ has no
fixed point in $\bd Z_0$ (Proposition~\ref{prop:EasyDichotomy}).

We claim that $\bd Z_0$ is Tits-isometric to the spherical building of the Lévi subgroup $\centra_G(\TT)$.
Indeed, we know from~\eqref{pt:algebraic:bd} that $\bd X$ is equivariantly isometric to $\bd X_{\mathrm{model}}$
and the building of $\centra_G(\TT)$ is characterised as the points at distance~$\pi/2$ from the boundary of the
$T$-invariant flat in $\bd X_{\mathrm{model}}$.

On the other hand, $\LL$ has maximal semi-simple rank in $\centra_G(\TT)$ by the choice of $\PP$ and therefore
cannot be contained in a proper parabolic subgroup of $\centra_G(\TT)$. This shows that $L$ has no fixed point
in $\bd Z_0$ and completes the proof.
\end{proof}

%================================================================================
\subsection{No branching geodesics}
Recall that in a geodesic metric space $X$, the \textbf{space of directions}\index{space of directions}
$\Sigma_x $ at a point $x$ is the completion of the space $\wt\Sigma_x$ of geodesic germs equipped with the
Alexandrov angle metric at $x$. If $X$ has \emph{uniquely} extensible geodesics, then $\wt\Sigma_x = \Sigma_x$.

\medskip

The following is a result of V.~Berestovskii~\cite{Berestovskii03} (we read it
in~\cite[\S\,3]{Berestovskii03_preprint}; it also follows from A.~Lytchak's arguments
in~\cite[\S\,4]{Lytchak_RigidityJoins}).

\begin{thm}\label{thm:sphere}
Let $X$ be a proper \cat space with uniquely extensible geodesics and $x\in X$. Then $\wt\Sigma_x = \Sigma_x$ is
isometric to a Euclidean sphere.\hfill\qedsymbol
\end{thm}

We use this result to establish the following.

\begin{prop}\label{prop:extensible_discrete}
Let $X$ be a proper \cat space with uniquely extensible geodesics. Then any totally disconnected closed subgroup
$D<\Isom(X)$ is discrete.
\end{prop}

\begin{proof}
There is some compact open subgroup $Q<D$, see~\cite[III \S\,4 No~6]{BourbakiTGI}. Let $x$ be a $Q$-fixed point.
The isometry group of $\Sigma_x$ is a compact Lie group by Theorem~\ref{thm:sphere} and thus the image of the
profinite group $Q$ in it is finite. Let thus $K<Q$ be the kernel of this representation, which is open. Denote
by $S(x,r)$ the $r$-sphere around $x$. The $Q$-equivariant ``visual'' map $S(x,r)\to \Sigma_x$ is a bijection by
unique extensibility. It follows that $K$ is trivial.
\end{proof}

We are now ready for:
\begin{proof}[End of proof of Theorem~\ref{thm:GeodesicallyComplete}]
Since the action is cocompact, it is minimal by Lemma~\ref{lem:cocompact:minimal}. The fact that extensibility
of geodesics is inherited by direct factors of the space follows from the characterisation of geodesics in
products, see~\cite[I.5.3(3)]{Bridson-Haefliger}. Each factor $X_i$ is thus a symmetric space in view of
Theorem~\ref{thm:algebraic}\eqref{pt:algebraic:homo}. By virtue of
Corollary~\ref{cor:BasicTotDisc}\eqref{pt:BasicTotDisc:completess}, the totally disconnected factors $D_j$ act
by semi-simple isometries.

Assume now that $X$ has uniquely extensible geodesics. For the same reason as before, this property is inherited
by each direct factor of the space. Thus each $D_j$ is discrete by Proposition~\ref{prop:extensible_discrete}.
\end{proof}

\begin{thm}\label{thm:nondiscrete:nonbranching}
Let $X$ be a proper irreducible \cat space with uniquely extensible geodesics. If $X$ admits a non-discrete
group of isometries with full limit set but no global fixed point at infinity, then $X$ is a symmetric space.
\end{thm}

\noindent The condition on fixed points at infinity is necessary in view of E.~Heintze's examples~\cite{Heintze74}
of negatively curved homogeneous manifolds which are not symmetric spaces. In fact these spaces consist of
certain simply connected soluble Lie groups endowed with a left-invariant negatively curved Riemannian metric.

\begin{prop}\label{prop:dimension:unique:geod}
Let $X$ be a proper \cat space with uniquely extensible geodesics. Then $\bd X$ has finite dimension.
\end{prop}

\begin{proof}
Let $x\in X$ and recall that by Berestovskii's result quoted in Theorem~\ref{thm:sphere} above, $\Sigma_x$ is
isometric to a Euclidean sphere. By definition of the Tits angle, the ``visual'' map $\bd X \to \Sigma_x$
associating to a geodesic ray its germ at $x$
is Tits-continuous (in fact, $1$-Lipschitz). It is furthermore injective (actually, bijective) by unique extensibility.
Therefore, the topological dimension of any \emph{compact} subset of $\bd X$ is bounded by the dimension of the sphere
$\Sigma_x$. The claim follows now from Kleiner's characterisation of the dimension of spaces with curvature bounded
above in terms of the topological dimension of compact subsets (Theorem~A in~\cite{Kleiner}).
\end{proof}

\begin{proof}[Proof of Theorem~\ref{thm:nondiscrete:nonbranching}]
By Lemma~\ref{lem:cocompact:minimal}, the action of $G:=\Isom(X)$ is minimal. In view of
Proposition~\ref{prop:dimension:unique:geod}, the boundary $\bd X$ is finite-dimensional.
Thus we can apply Theorem~\ref{thm:Decomposition} and Addendum~\ref{addendum}.
Since $X$ is irreducible and non-discrete, Proposition~\ref{prop:extensible_discrete} implies that $G$ is an almost
connected simple Lie group (unless $X=\RR$, in which case $X$ is indeed a symmetric space).
We conclude by Theorem~\ref{thm:algebraic}.
\end{proof}

%================================================================================
\subsection{No open stabiliser at infinity}\label{sec:NoOpenStabiliser}
The following statement sums up some of the preceding considerations:

\begin{cor}\label{cor:NoOpenStabiliser}
Let $X$ be a proper geodesically complete \cat space without Euclidean factor such that some closed subgroup $G
< \Isom(X)$ acts cocompactly. Suppose that no open subgroup of $G$ fixes a point at infinity. Then we have the
following:
\begin{enumerate}
\item $X$ admits a canonical equivariant splitting
$$X \cong\ X_1\times \cdots \times X_p \times  Y_1\times \cdots \times Y_q$$
where each $X_i$ is a symmetric space and each $Y_j$ possesses a $G$-equivariant locally finite decomposition
into compact convex cells.

\item $G$ possesses hyperbolic elements.

\item Every compact subgroup of $G$ is contained in a maximal one; the maximal compact subgroups fall into
finitely many conjugacy classes.

\item $\QZ(G) = 1$; in particular $G$ has no non-trivial discrete normal subgroup.

\item $\soc(G^*)$ is a direct product of $p+q$ non-discrete characteristically simple groups.
\end{enumerate}
\end{cor}

\begin{proof}
(i) Follows from Theorem~\ref{thm:GeodesicallyComplete} and Corollary~\ref{cor:CellularDecomposition}.

\smallskip \noindent (ii)  Clear from Corollary~\ref{cor:exist:nontorsion}.

\smallskip \noindent (iii) and (iv) Immediate from (i) and Theorem~\ref{thm:TotDiscIrreducible}(i) and~(ii).

\smallskip \noindent (v) Follows from (i), (iv) and Proposition~\ref{prop:socle}.
\end{proof}

%\begin{proof}[Proof of Theorems~\ref{thm:NoOpenStabiliser:space} and~\ref{thm:NoOpenStabiliser:group}]
%It suffices to observe that the above proof of~(i) is immune to Euclidean factors (as is the statement of~(i)
%itself actually).
%\end{proof}

%================================================================================
\subsection{Cocompact stabilisers at infinity}
We undertake the proof of Theorem~\ref{thm:cocompact:stabilisers} which describes isometrically any
geodesically complete proper \cat space such that the stabiliser of every point at infinity
acts cocompactly.

\begin{remark}
(i)~The formulation of Theorem~\ref{thm:cocompact:stabilisers} allows for symmetric spaces of Euclidean type.
(ii)~A \textbf{Bass--Serre tree}\index{Bass--Serre tree} is a tree admitting an edge-transitive automorphism
group; in particular, it is regular or bi-regular (the regular case being a special case of Euclidean
buildings).
\end{remark}

\begin{lem}\label{lem:antipode}
Let $X$ be a proper \cat space such that the stabiliser of every point at infinity acts cocompactly on $X$. For
any $\xi\in\bd X$, the set of $\eta\in\bd X$ with $\tangle\xi\eta =\pi$ is contained in a single orbit under
$\Isom(X)$.
\end{lem}

\begin{proof}
Write $G=\Isom(X)$. In view of Proposition~\ref{prop:fixed:cocompact} applied to $G_\xi$, it suffices to prove
that the $G$-orbit of any such $\eta$ contains a point opposite to $\xi$. By definition of the Tits angle, there
is a sequence $\{x_n\}$ in $X$ such that $\aangle{x_n}\xi\eta$ tends to~$\pi$. Since $G_\xi$ acts cocompactly,
it contains a sequence $\{g_n\}$ such that, upon extracting, $g_n x_n$ converges to some $x\in X$ and $g_n \eta$
to some $\eta'\in\bd X$. The angle semi-continuity arguments given in the proof of Proposition~II.9.5(3)
in~\cite{Bridson-Haefliger} show that $\aangle x\xi{\eta'}=\pi$, recalling that all $g_n$ fix $\xi$. This means
that there is a geodesic $\sigma:\RR\to X$ through $x$ with $\sigma(-\infty)=\xi$ and $\sigma(\infty)=\eta'$. On
the other hand, since $G_\eta$ is cocompact in $G$, the $G$-orbit of $\eta$ is closed in the c\^one topology.
This means that there is $g\in G$ with $\eta'=g\eta$, as was to be shown.
\end{proof}

We shall need another form of angle rigidity (compare Proposition~\ref{prop:angle_rigidity}), this time for Tits
angles.\index{angle!rigidity}

\begin{prop}\label{prop:discrete:orbits}
Let $X$ be a geodesically complete proper \cat space, $G<\Isom(X)$ a closed totally disconnected subgroup and
$\xi\in\bd X$. If the stabiliser $G_\xi$ acts cocompactly on $X$, then the $G$-orbit of $\xi$ is discrete in the
Tits topology.
\end{prop}

\begin{proof}
Suppose for a contradiction that there is a sequence $\{g_n\}$ such that $g_n\xi\neq\xi$ for all $n$ but
$\tangle{g_n \xi}\xi$ tends to zero. Since $G_\xi$ is cocompact, we can assume that $g_n$ converges in $G$;
since the Tits topology is finer than the c\^one topology for which the $G$-action is continuous, the limit of
$g_n$ must fix $\xi$ and we can therefore assume $g_n\to 1$. Let $B\se X$ be an open ball large enough so that
$G_\xi.B = X$. Since by Lemma~\ref{lem:cocompact:minimal} we can apply Theorem~\ref{thm:smooth}, there is no
loss of generality in assuming that each $g_n$ fixes $B$ pointwise.

Let $c:\RR_+\to X$ be a geodesic ray pointing towards $\xi$ with $c(0)\in B$. For each $n$ there is $r_n>0$ such
that $c$ and $g_n c$ branch at the point $c(r_n)$. In particular, $g_n$ fixes $c(r_n)$ but not $c(r_n+\vareps)$
no matter how small $\vareps>0$. We now choose $h_n\in G_\xi$ such that $x_n:=h_n c(r_n) \in B$ and notice that
the sequence $k_n:=h_n g_n h_n\inv$ is bounded since $k_n$ fixes $x_n$. We can therefore assume upon extracting
that it converges to some $k\in G$; in view of Theorem~\ref{thm:smooth}, we can further assume that all $k_n$
coincide with $k$ on $B$ and in particular $k$ fixes all $x_n$. Since $\tangle{k_n\xi}\xi = \tangle{g_n\xi}\xi$,
we also have $k\in G_\xi$. Considering any given $n$, it follows now that $k$ fixes the ray from $x_n$ to $\xi$.
Thus $k_n$ fixes an initial segment of this ray at $x_n$. This is equivalent to $g_n$ fixing an initial segment
at $c(r_n)$ of the ray from $c(r_n)$ to $\xi$, contrary to our construction.
\end{proof}

Here is a first indication that our spaces might resemble symmetric spaces or Euclidean buildings:

\begin{prop}\label{prop:spheres_at_infinity}
Let $X$ be a proper \cat space such that that the stabiliser of every point at infinity
acts cocompactly on $X$. Then any point at infinity is contained in an isometrically embedded
standard $n$-sphere in $\bd X$, where $n=\dim\bd X$.
\end{prop}

\begin{proof}
Let $\eta\in\bd X$. There is some standard $n$-sphere $S$ isometrically embedded in $\bd X$ because $X$ is
cocompact (Theorem~C in~\cite{Kleiner}). By Lemma~3.1 in~\cite{BalserLytchak_Centers}, there is $\xi\in S$ with
$\tangle\xi\eta=\pi$. Let $\teta\in S$ be the antipode in $S$ of $\xi$. In view of Lemma~\ref{lem:antipode},
there is an isometry sending $\teta$ to $\eta$. The image of $S$ contains $\eta$.
\end{proof}

We need one more fact for Theorem~\ref{thm:cocompact:stabilisers}. The boundary of a \cat space
need not be complete, regardless of the geodesic completeness of the space itself;
however, this is the case in our situation in view of Proposition~\ref{prop:spheres_at_infinity}:

\begin{cor}\label{cor:bnd_geod_complete}
Let $X$ be a proper \cat space such that that the stabiliser of every point at infinity
acts cocompactly on $X$. Then $\bd X$ is geodesically complete.
\end{cor}

\begin{proof}
Suppose for a contradiction that some Tits-geodesic ends at $\xi\in\bd X$ and let $B\se \bd X$ be a small convex
Tits-neighbourhood of $\xi$; in particular, $B$ is contractible. Since by Proposition~\ref{prop:spheres_at_infinity} there is an $n$-sphere
through $\xi$ for $n=\dim\bd X$, the relative homology $\mathrm{H}_n(B, B\setminus \{\xi\})$ is non-trivial.
Our assumption implies that $B\setminus \{\xi\}$ is contractible by using the geodesic contraction to some point $\eta\in B\setminus \{\xi\}$
on the given geodesic ending at $\xi$. This implies $\mathrm{H}_n(B, B\setminus \{\xi\})=0$, a contradiction.
(This argument is adapted from~\cite[II.5.12]{Bridson-Haefliger}.)
\end{proof}

\begin{proof}[End of proof of Theorem~\ref{thm:cocompact:stabilisers}]
We shall use below that product decompositions preserve geodesic completeness (this follows \emph{e.g.}
from~\cite[I.5.3(3)]{Bridson-Haefliger}). We can reduce to the case where $X$ has no Euclidean factor. By
Lemma~\ref{lem:cocompact:minimal}, the group $G=\Isom(X)$ as well as all stabilisers of points at infinity act
minimally. In particular, Proposition~\ref{prop:StabIsMinimal} ensures that $G$ has no fixed point at infinity
and we can apply Theorem~\ref{thm:Decomposition} and Addendum~\ref{addendum}. Therefore, we can from now on
assume that $X$ is irreducible. If the identity component $G^\circ$ is non-trivial, then
Theorem~\ref{thm:GeodesicallyComplete} (see also Theorem~\ref{thm:algebraic}(iii)) ensures that $X$ is a
symmetric space, and we are done. We assume henceforth that $G$ is totally disconnected.

For any $\xi\in\bd X$, the collection $\mathrm{Ant}(\xi)=\{\eta:\tangle\xi\eta =\pi\}$ of
\textbf{antipodes}\index{antipode} is contained in a $G$-orbit by Lemma~\ref{lem:antipode} and hence is
Tits-discrete by Proposition~\ref{prop:discrete:orbits}. This discreteness and the geodesic completeness of the
boundary (Corollary~\ref{cor:bnd_geod_complete}) are the assumptions needed for Proposition~4.5
in~\cite{Lytchak_RigidityJoins}, which states that $\bd X$ is a building. Since $X$ is irreducible, $\bd X$ is
not a (non-trivial) spherical join, see Theorem~II.9.24 in~\cite{Bridson-Haefliger}. Thus, if this building has
non-zero dimension, we conclude from the main result of~\cite{Leeb} that $X$ is a Euclidean building of higher
rank.

If on the other hand $\bd X$ is zero-dimensional, then we claim that it is homogeneous under $G$. Indeed, we
know already that for any given $\xi\in\bd X$, the set $\mathrm{Ant}(\xi)$ lies in a single orbit. Since in the
present case $\mathrm{Ant}(\xi)$ is simply $\bd X \setminus\{\xi\}$, the claim follows from the fact that $G$
has no fixed point at infinity.

\medskip%
We have to show that $X$ is an edge-transitive tree. To this end, consider any point $x \in X$.
The isotropy group $G_x$ is open by Theorem~\ref{thm:smooth}. In particular, since $G$ acts transitively on $\bd
X$ and since $G_\xi$ is cocompact, it follows that $G_x$ has finitely many orbits in $\bd X$. Let $\rho_1,
\dots, \rho_k$ be geodesic rays emanating from $x$ and pointing towards boundary points which provide a complete
set of representatives for the $G_x$-orbits. For $r
>0$ sufficiently large, the various intersections of the rays $\rho_1, \dots, \rho_k$ with the $r$-sphere
$S_r(x)$ centred at $x$ forms a set of $k$ distinct points. This set is a fundamental domain for the
$G_{x}$-action on $S_r(x)$. Since $G_{x}$ has discrete orbits on $S_r(x)$, we deduce from
Theorem~\ref{thm:smooth} that the sphere $S_r(x)$ is finite. Since this holds for any $r>0$ sufficiently large
and any $x \in X$, it follows that every sphere in $X$ is finite. This implies that $X$ is $1$-dimensional (see
\cite{Kleiner}). In other words $X$ is a metric tree. We denote by $V$ the set of branch points which we shall
call the vertices. It remains to show that $G$ has at most two orbits on $V$.

Given $\xi' \in \bd X$, let $\beta_{\xi'}: G_{\xi'} \to \RR$ denote the Busemann character centred at $\xi'$
(see \S\,\ref{sec:notation}). Since $X$ is a cocompact tree, it follows that $\beta_{\xi'}$ has discrete image. Let $g \in
G_{\xi'}$ be an element such that $ \beta_{\xi'}(g)$ is positive and minimal. Then $g$ is hyperbolic and
translates a geodesic line $L$. Let $\xi''$ denote the endpoint of $L$ distinct from $\xi$.

Let $v \in L$ be any vertex. We denote by $e'$ and $e''$ the edges of $L$ containing $v$ and pointing
respectively to $\xi'$ and $\xi''$. Give any edge $e$ containing $v$ with $e' \neq e \neq e''$, we prolong $e$
to a geodesic ray $\rho$ whose intersection with $L$ is reduced to $\{v\}$. Since $G_{\xi'}$ is transitive on
$\bd X \setminus\{\xi'\}$ there exists $g' \in G_{\xi'}$ such that $g'.\xi'' = \rho(\infty)$. Upon pre-composing
$g'$ with a suitable power of $g$, we may assume that $\beta_{\xi'}(g')=0$. In other words $g'$ fixes $v$. This
shows that $G_{\xi', v}$ is transitive on the edges containing $v$ and different from $e'$.

The same argument with $\xi'$ and $\xi''$ interchanged shows that $G_{\xi'', v}$ is transitive on the edges
containing $v$ and different from $e''$. In particular $G_v$ is transitive on the edges containing $v$.

A straightforward induction on the distance to $v$ now shows that for any vertex $w \in V$, the isotropy group
$G_w$ is transitive on the edges containing $w$. This implies  that $G$ is indeed edge-transitive.
\end{proof}

%%%%%%%%%%%%%%%%%%%%%%%%%%%%%%%%%%%%%%%%%%%%%%%%%%%%%%%%%%%%%%%%%%%%%%%%%%%%%%
\section{A few cases of \cat superrigidity}\label{sec:few:super}
%%%%%%%%%%%%%%%%%%%%%%%%%%%%%%%%%%%%%%%%%%%%%%%%%%%%%%%%%%%%%%%%%%%%%%%%%%%%%%
This Section demonstrates that certain forms of superrigidity can be obtained by combining the structure results of this paper with
known superrigidity techniques. Much more general results will be established in the companion paper~\cite{Caprace-Monod_discrete}.

\subsection{\cat superrigidity for some classical non-uniform lattices}\label{sec:NaiveSuperrigidity}
Let $\Gamma$ be a non-uniform lattice in a simple (real) Lie group $G$ of rank at least~$2$.
By~\cite[Theorem~2.15]{LMR_nondistorsion}, unipotent elements of $\Gamma$ are exponentially distorted. This means
that, with respect to any finitely generating set of $\Gamma$, the word length of  $|u^n|$ is an $O(\log n)$
when $u$ is a unipotent. More generally an element $u$ is called \emph{distorted} if $|u^n|$ is sublinear.
If $\Gamma$ is virtually boundedly generated by unipotent elements, one can therefore
apply the following fixed point principle:

\begin{lem}\label{lem:FixedPointSemisimpleActions}
Let $\Gamma$ be a group which is virtually boundedly generated by distorted elements. Then any isometric
$\Gamma$-action on a complete \cat space such that elements of zero translation length are elliptic
has a global fixed point.
\end{lem}

\begin{proof}
For any $\Gamma$-action on a \cat space, the translation length of a distorted element is zero. Thus every such
element has a fixed point; the assumption on $\Gamma$ now implies that all orbits are bounded, thus providing
a fixed point~\cite[II.2.8(1)]{Bridson-Haefliger}.
\end{proof}

Bounded generation is a strong property, which conjecturally holds for all (non-uniform) lattices of a higher
rank semi-simple Lie group. It is known to hold for arithmetic groups in split or quasi-split algebraic
groups of a number field $K$ of $K$-rank~$\geq 2$ by~\cite{Tavgen}, as well as in a few cases of isotropic but
non-quasi-split groups~\cite{ErovenkoRapinchuk}.

As noticed in a conversation with Sh.~Mozes, Lemma~\ref{lem:FixedPointSemisimpleActions} yields the following
elementary superrigidity statement.

\begin{prop}\label{prop:superrigidity:SLnQp}
Let $\Lambda = \SL_n(\ZZ[\frac{1}{p_1\cdots p_k}])$ with $n \geq 3$ and $p_i$ distinct primes
and set $H=\SL_n(\QQ_{p_1})\times\cdots\times \SL_n(\QQ_{p_k})$.

Given any isometric $\Lambda$-action on any complete \cat space such that every element of zero
translation length is elliptic, there exists a $\Lambda$-invariant closed convex subspace on which
the given action extends uniquely to a continuous $H$-action by isometries.
\end{prop}

\begin{proof}
Let $X$ be a complete \cat space endowed with a $\Lambda$-action as in the statement. The subgroup
$\Gamma=\SL_n(\ZZ)<\Lambda$ fixes a point by Lemma~\ref{lem:FixedPointSemisimpleActions}.
The statement now follows because $\Gamma$ is the intersection of $\Lambda$ with the \emph{open} subgroup
$\SL_n(\ZZ_{p_1})\times\cdots\times \SL_n(\ZZ_{p_k})$ of $H$; for later use, we isolate this elementary fact as
Lemma~\ref{lem:elementary:superrigidity} below.
\end{proof}

\begin{lem}\label{lem:elementary:superrigidity}
Let $H$ be a topological group, $U<H$ an open subgroup, $\Lambda<H$ a dense subgroup and $\Gamma=\Lambda\cap U$.
Any $\Lambda$-action by isometries on a complete \cat space with a $\Gamma$-fixed point admits a $\Lambda$-invariant
closed convex subspace on which the action extends continuously to $H$.
\end{lem}

\begin{proof}
Let $X$ be the \cat space and $x_0\in X$ a $\Gamma$-fixed point. For any finite subset $F\se \Lambda$, let $Y_F\se X$
be the closed convex hull of $F x_0$. The closed convex hull $Y$ of $\Lambda x_0$ is the closure of the union $Y_\infty$ of the
directed family $\{Y_F\}$. Therefore, since the action is isometric and $Y$ is complete, it suffices to show that the
$\Lambda$-action on $Y_\infty$ is continuous for the topology induced on $\Lambda$ by $H$.
Equivalently, it suffices to prove that all orbital maps $\Lambda\to Y_\infty$ are continuous at $1\in \Lambda$.
This is the case even for the discrete topology on $Y_\infty$ because the pointwise fixator of each $Y_F$ is an
intersection of finitely many conjugates of $\Gamma$, the latter being open by definition.
\end{proof}

The same arguments as below show that Theorem~\ref{thm:superrigidity} holds for any lattice of a higher-rank
semi-simple Lie group which is boundedly generated by distorted elements (and accordingly Theorem~\ref{thm:superrigidity_bis}
generalises to suitable (S-)arithmetic groups).

\begin{proof}[Proof of Theorems~\ref{thm:superrigidity} and~\ref{thm:superrigidity_bis}]
We start with the case $\Gamma= \SL_n(\ZZ)$.
By Theorem~\ref{thm:GeodesicallyComplete}, we obtain a closed convex subspace $X'$ which splits as a direct
product
$$X' \cong X_1 \times \dots \times X_p \times Y_0 \times Y_1 \times \dots \times Y_q$$
in an $\Isom(X')$-equivariant way, where $Y_0 \cong \RR^n$ is the Euclidean factor. Each totally disconnected
factor $D_i$ of $\Isom(X')^*$ acts by semi-simple isometries on the corresponding factor $Y_i$ of $X'$ by
Corollary~\ref{cor:BasicTotDisc}. Therefore, by Lemma~\ref{lem:FixedPointSemisimpleActions} for each $i=0,
\dots, q$, the induced $\Gamma$-action on $Y_i$ has a global fixed point, say $y_i$. In other words $\Gamma$
stabilises the closed convex subset
$$Z := X_1 \times \dots \times X_p \times \{y_0\} \times \dots \times \{y_q\} \se X.$$
Note that the isometry group of $Z$ is an almost connected semi-simple real Lie group $L$. Combining
Lemma VII.5.1 and Theorems VII.5.15 and VII.6.16 from~\cite{Margulis}, it follows that the Zariski closure of the image
of $\Gamma$ in $L$ is a commuting product $L_1 . L_2$, where $L_1$ is compact, such that the corresponding
homomorphism $\Gamma \to L_2$ extends to a continuous homomorphism $G \to L_2$. We define $Y \subseteq Z$ as the
fixed point set of $L_1$. Now $L_2$, and hence $\Gamma$, stabilises $Y$. Therefore the continuous homomorphism
$G \to L_2$ yields a $G$-action on $Y$ which extends the given $\Gamma$-action, as desired.

Applying Theorem~\ref{thm:algebraic} point~\eqref{pt:algebraic:sub} to the pair $L_2<L$ acting on $Z$, we find in
particular that $L_2$ has no fixed point at infinity in $Y$. Thus, upon replacing $Y$ by a subspace, it is
$L_2$-minimal. Now Theorem~\dref{thm:density} (which is completely independent of the present considerations)
implies that the $\Gamma$- and $G$-actions on $Y$ are minimal and
without fixed point in $\bd Y$ (although there might be fixed points in $\bd X$).

\medskip

Turning to Theorem~\ref{thm:superrigidity_bis}, the only change is that one replaces
Lemma~\ref{lem:FixedPointSemisimpleActions} by Proposition~\ref{prop:superrigidity:SLnQp}.
\end{proof}

%================================================================================
\subsection{\cat superrigidity for irreducible lattices in products}
The aim of this section is to state a version of the superrigidity theorem~\cite[Theorem~6]{Monod_superrigid}
with \cat targets. The original statement from \emph{loc.\ cit.} concerns actions of lattices on
\emph{arbitrary} \cat spaces, with reduced unbounded image. The following statement shows that, when the
underlying \cat space is nice enough, the assumption on the action can be considerably weakened.

We recall for the statement that any isometric action on a proper \cat space without globel fixed point
at infinity admits a \emph{canonical} minimal non-empty closed convex invariant subspace, see Remarks~39
in~\cite{Monod_superrigid}.

\begin{thm}\label{thm:Monod:superrigidity}
Let $\Gamma$ be an irreducible uniform (or square-integrable weakly cocompact) lattice in a product $G = G_1
\times \dots  \times G_n$ of $n \geq 2$ locally compact $\sigma$-compact groups. Let $X$ be a proper \cat space
with finite-dimensional boundary.

Given any $\Gamma$-action on $X$ without fixed point at infinity, if the
canonical $\Gamma$-minimal subset $Y \subseteq X$
has no Euclidean factor, then the $\Gamma$-action on $Y$ extends to a
continuous $G$-action by isometries.
\end{thm}

\begin{remark}
Although the above condition on the Euclidean factor in the $\Gamma$-minimal subspace $Y$ might seem awkward, it
cannot be avoided, as illustrated by Example~64 in~\cite{Monod_superrigid}. Notice however that if $\Gamma$ has
the property that any isometric action on a finite-dimensional Euclidean space has a global fixed (for example
if $\Gamma$ has Kazhdan's property (T)), then any minimal $\Gamma$-invariant subspace has no Euclidean factor.
\end{remark}

\begin{proof}[Proof of Theorem~\ref{thm:Monod:superrigidity}]
Let $Y \subseteq X$ be the canonical subspace recalled above. Then $\Isom(Y)$
acts minimally on $Y$, without fixed point at infinity. In particular we may apply
Theorem~\ref{thm:Decomposition} and Addendum~\ref{addendum}. In order to show that the $\Gamma$-action on $Y$
extends to a continuous $G$-action, it is sufficient to show that the induced $\Gamma$-action on each
irreducible factor of $Y$ extends to a continuous $G$-action, factoring through some $G_i$. But the induced
$\Gamma$-action on each irreducible factor of $Y$ is reduced by Corollary~\ref{cor:ReducedAction}. Thus the
result follows from~\cite[Theorem~6]{Monod_superrigid}.
\end{proof}

%=========================================================================================================
%\bibliographystyle{amsalpha}
\def\cprime{$'$}
\providecommand{\bysame}{\leavevmode\hbox to3em{\hrulefill}\thinspace}

%\bibliography{../IsomCAT0}

\begin{thebibliography}{LMR00}

\bibitem[AB98a]{AB98}
Scot Adams and Werner Ballmann, \emph{Amenable isometry groups of {H}adamard
  spaces}, Math. Ann. \textbf{312} (1998), no.~1, 183--195.

\bibitem[AB98b]{Alexander-Bishop98}
Stephanie~B. Alexander and Richard~L. Bishop, \emph{Warped products of
  {H}adamard spaces}, Manuscripta Math. \textbf{96} (1998), no.~4, 487--505.

\bibitem[AB04]{Alexander-Bishop04}
\bysame, \emph{Curvature bounds for warped products of metric spaces}, Geom.
  Funct. Anal. \textbf{14} (2004), no.~6, 1143--1181.

\bibitem[Ale57]{Alexandrov}
Alexander~D. Alexandrow, \emph{\"{U}ber eine {V}erallgemeinerung der
  {R}iemannschen {G}eometrie}, Schr. Forschungsinst. Math. \textbf{1} (1957),
  33--84.

\bibitem[Bal95]{BallmannLN}
Werner Ballmann, \emph{Lectures on spaces of nonpositive curvature}, DMV
  Seminar, vol.~25, Birkh\"auser Verlag, Basel, 1995, With an appendix by Misha
  Brin.

\bibitem[Ber]{Berestovskii03_preprint}
Valeri\u\i~N. Berestovski\u\i, \emph{{B}usemann spaces of {A}leksandrov
  curvature bounded above}, Max-Planck preprint Nr.~74/2001.

\bibitem[Ber02]{Berestovskii03}
\bysame, \emph{Busemann spaces with upper-bounded {A}leksandrov curvature},
  Algebra i Analiz \textbf{14} (2002), no.~5, 3--18.

\bibitem[BGP92]{Burago-Gromov-Perelman}
Yuri Burago, Mikha{\"\i}l Gromov, and Gregory Perel{\cprime}man, \emph{A. {D}.
  {A}leksandrov spaces with curvatures bounded below}, Uspekhi Mat. Nauk
  \textbf{47} (1992), no.~2(284), 3--51, 222.

\bibitem[BH99]{Bridson-Haefliger}
Martin~R. Bridson and Andr\'e Haefliger, \emph{{Metric spaces of non-positive
  curvature}}, {Grundlehren der Mathematischen Wissenschaften 319, Springer,
  Berlin}, 1999.

\bibitem[BL05]{BalserLytchak_Centers}
Andreas Balser and Alexander Lytchak, \emph{Centers of convex subsets of
  buildings}, Ann. Global Anal. Geom. \textbf{28} (2005), no.~2, 201--209.

\bibitem[BM96]{BM96}
Marc Burger and Shahar Mozes, \emph{{{\rm CAT(-1)}}-spaces, divergence groups
  and their commensurators}, J. Amer. Math. Soc. \textbf{9} (1996), 57--93.

\bibitem[BM00]{Burger-Mozes1}
\bysame, \emph{Groups acting on trees: from local to global structure}, Inst.
  Hautes \'Etudes Sci. Publ. Math. (2000), no.~92, 113--150 (2001).

\bibitem[BO69]{Bishop-ONeill}
Richard~L. Bishop and Barrett O'Neill, \emph{Manifolds of negative curvature},
  Trans. Amer. Math. Soc. \textbf{145} (1969), 1--49.

\bibitem[Bou68]{Bourbaki_Lie456}
Nicolas Bourbaki, \emph{\'{E}l\'ements de math\'ematique. {F}asc. {XXXIV}.
  {G}roupes et alg\`ebres de {L}ie. {C}hapitre {IV}: {G}roupes de {C}oxeter et
  syst\`emes de {T}its. {C}hapitre {V}: {G}roupes engendr\'es par des
  r\'eflexions. {C}hapitre {VI}: syst\`emes de racines}, Actualit\'es
  Scientifiques et Industrielles, No. 1337, Hermann, Paris, 1968.

\bibitem[Bou71]{BourbakiTGI}
\bysame, \emph{\'{E}l\'ements de math\'ematique. {T}opologie g\'en\'erale.
  {C}hapitres 1 \`a 4}, Hermann, Paris, 1971.

\bibitem[Cap07]{CapraceTD}
Pierre-Emmanuel Caprace, \emph{Amenable groups and hadamard spaces with a
  totally disconnected isometry group}, Preprint, available at
  \texttt{arXiv:0705.1980v1}, 2007.

\bibitem[Che99]{Chen99}
Chien-Hsiung Chen, \emph{Warped products of metric spaces of curvature bounded
  from above}, Trans. Amer. Math. Soc. \textbf{351} (1999), no.~12, 4727--4740.

\bibitem[CK00]{Croke-Kleiner}
Christopher~B. Croke and Bruce Kleiner, \emph{Spaces with nonpositive curvature
  and their ideal boundaries}, Topology \textbf{39} (2000), no.~3, 549--556.

\bibitem[CL08]{Caprace-Lytchak}
Pierre-Emmanuel Caprace and Alexander Lytchak, \emph{At infinity of
  finite-dimensional {CAT}(0) spaces}, Preprint, 2008.

\bibitem[CM08a]{Caprace-Monod_monolith}
Pierre-Emmanuel Caprace and Nicolas Monod, \emph{Decomposing locally compact
  groups into simple pieces}, Preprint, arXiv:math/0811.4101, 2008.

\bibitem[CM08b]{Caprace-Monod_discrete}
\bysame, \emph{Isometry groups of non-positively curved spaces: Discrete
  subgroups}, Preprint, 2008.

\bibitem[dR52]{deRham52}
Georges de~Rham, \emph{Sur la reductibilit\'e d'un espace de {R}iemann},
  Comment. Math. Helv. \textbf{26} (1952), 328--344.

\bibitem[ER06]{ErovenkoRapinchuk}
Igor~V. Erovenko and Andrei~S. Rapinchuk, \emph{Bounded generation of
  {$S$}-arithmetic subgroups of isotropic orthogonal groups over number
  fields}, J. Number Theory \textbf{119} (2006), no.~1, 28--48.

\bibitem[Far08]{Farb:Helly}
Benson Farb, \emph{Group actions and {Helly}'s theorem}, Preprint, 2008.

\bibitem[FL06]{FoertschLytchak06}
Thomas Foertsch and Alexander Lytchak, \emph{The de {R}ham decomposition
  theorem for metric spaces}, GAFA (to appear), arXiv:math/0605419v1, 2006.

\bibitem[FNS06]{Fujiwara_et_al}
Koji Fujiwara, Koichi Nagano, and Takashi Shioya, \emph{Fixed point sets of
  parabolic isometries of {CAT}(0)-spaces}, Comment. Math. Helv. \textbf{81}
  (2006), no.~2, 305--335.

\bibitem[GW71]{Gromoll-Wolf}
Detlef Gromoll and Joseph~A. Wolf, \emph{Some relations between the metric
  structure and the algebraic structure of the fundamental group in manifolds
  of nonpositive curvature}, Bull. Amer. Math. Soc. \textbf{77} (1971),
  545--552.

\bibitem[Hei74]{Heintze74}
Ernst Heintze, \emph{On homogeneous manifolds of negative curvature}, Math.
  Ann. \textbf{211} (1974), 23--34.

\bibitem[HLS00]{Hummel-Lang-Schroeder}
Christoph Hummel, Urs Lang, and Viktor Schroeder, \emph{Convex hulls in
  singular spaces of negative curvature}, Ann. Global Anal. Geom. \textbf{18}
  (2000), no.~2, 191--204.

\bibitem[Kle99]{Kleiner}
Bruce Kleiner, \emph{The local structure of length spaces with curvature
  bounded above}, Math. Z. \textbf{231} (1999), no.~3, 409--456.

\bibitem[KM99]{KarlssonMargulis}
Anders Karlsson and Gregory~A. Margulis, \emph{A multiplicative ergodic theorem
  and nonpositively curved spaces}, Comm. Math. Phys. \textbf{208} (1999),
  no.~1, 107--123.

\bibitem[Lee00]{Leeb}
Bernhard Leeb, \emph{A characterization of irreducible symmetric spaces and
  {E}uclidean buildings of higher rank by their asymptotic geometry}, Bonner
  Mathematische Schriften, 326, Universit\"at Bonn Mathematisches Institut,
  Bonn, 2000.

\bibitem[LMR00]{LMR_nondistorsion}
Alexander Lubotzky, Shahar Mozes, and Madabusi~Santanam Raghunathan, \emph{The
  word and {R}iemannian metrics on lattices of semisimple groups}, Inst. Hautes
  \'Etudes Sci. Publ. Math. (2000), no.~91, 5--53 (2001).

\bibitem[LY72]{Lawson-Yau}
H.~Blaine Lawson, Jr. and Shing-Tung Yau, \emph{Compact manifolds of
  nonpositive curvature}, J. Differential Geometry \textbf{7} (1972), 211--228.

\bibitem[Lyt05]{Lytchak_RigidityJoins}
Alexander Lytchak, \emph{Rigidity of spherical buildings and joins}, Geom.
  Funct. Anal. \textbf{15} (2005), no.~3, 720--752.

\bibitem[Mar91]{Margulis}
Gregory~A. Margulis, \emph{Discrete subgroups of semisimple {L}ie groups},
  Ergebnisse der Mathematik und ihrer Grenzgebiete (3) [Results in Mathematics
  and Related Areas (3)], vol.~17, Springer-Verlag, Berlin, 1991.

\bibitem[Mon01]{Monod_LN}
Nicolas Monod, \emph{Continuous bounded cohomology of locally compact groups},
  Lecture Notes in Mathematics, vol. 1758, Springer-Verlag, Berlin, 2001.

\bibitem[Mon06]{Monod_superrigid}
\bysame, \emph{Superrigidity for irreducible lattices and geometric splitting},
  J. Amer. Math. Soc. \textbf{19} (2006), no.~4, 781--814.

\bibitem[Pra77]{Prasad77}
Gopal Prasad, \emph{Strong approximation for semi-simple groups over function
  fields}, Ann. of Math. (2) \textbf{105} (1977), no.~3, 553--572.

\bibitem[Swe99]{Swenson99}
Eric~L. Swenson, \emph{A cut point theorem for {${\rm CAT}(0)$} groups}, J.
  Differential Geom. \textbf{53} (1999), no.~2, 327--358.

\bibitem[Tav90]{Tavgen}
Oleg~I. Tavgen{\cprime}, \emph{Bounded generability of {C}hevalley groups over
  rings of {$S$}-integer algebraic numbers}, Izv. Akad. Nauk SSSR Ser. Mat.
  \textbf{54} (1990), no.~1, 97--122, 221--222.

\bibitem[WM07]{Witte07}
Dave Witte~Morris, \emph{{Bounded generation of $\text{SL}(n,A)$ (after D.
  Carter, G. Keller, and E. Paige).}}, New York J. Math. \textbf{13} (2007),
  383--421.

\end{thebibliography}
%\printindex

\end{document}